\documentclass[11pt]{article}
\usepackage[latin1]{inputenc}
\usepackage{amsmath}
\usepackage{amsfonts}
\usepackage{amssymb, amsthm, stmaryrd}
\usepackage{graphicx}
\usepackage{xcolor}
\usepackage{multicol}
\usepackage{soul}
\usepackage[normalem]{ulem}
\usepackage[pagewise, displaymath, mathlines]{lineno}
\newtheorem{theorem}{Theorem}[section]
\newtheorem{conjecture}[theorem]{Conjecture}
\newtheorem{proposition}[theorem]{Proposition}
\newtheorem{problem}{Problem}
\newtheorem{corollary}[theorem]{Corollary}
\newtheorem{lemma}[theorem]{Lemma}
\newtheorem{remark}[theorem]{Remark}
\newtheorem{example}[theorem]{Example}

\newcommand{\Z}{\mbox{${\mathbb Z}$}}

\title{Zero-sum subsequences in bounded-sum $\{-1, 1\}$-sequences}

\begin{document}

\maketitle

\begin{center}

\begin{multicols}{2}

Yair Caro\\[1ex]
{\small Dept. of Mathematics \\
University of Haifa-Oranim\\
Tivon 36006, Israel\\
yacaro@kvgeva.org.il}

\columnbreak

Adriana Hansberg\\[1ex]
{\small Instituto de Matem\'aticas\\
UNAM Juriquilla\\
Quer\'etaro, Mexico\\
ahansberg@im.unam.mx}\\[2ex]

\end{multicols}

Amanda Montejano\\[1ex]
{\small UMDI, Facultad de Ciencias\\
UNAM Juriquilla\\
Quer\'etaro, Mexico\\
amandamontejano@ciencias.unam.mx}\\[4ex]

\end{center}

\begin{abstract}
The following result gives the flavor of this paper:
Let $t$, $k$ and $q$ be integers such that $q\geq 0$, $0\leq t < k$ and $t \equiv k \,({\rm mod}\, 2)$, and let $s\in [0,t+1]$ be the unique integer satisfying $s \equiv q + \frac{k-t-2}{2} \,({\rm mod} \, (t+2))$. Then for any integer $n$ such that
\[n \ge \max\left\{k,\frac{1}{2(t+2)}k^2 + \frac{q-s}{t+2}k - \frac{t}{2} + s\right\}\]
and any function $f:[n]\to \{-1,1\}$ with $|\sum_{i=1}^nf(i)| \le q$, there is a set $B \subseteq [n]$ of $k$ consecutive integers with $|\sum_{y\in B}f(y)| \le t$. Moreover, this bound is sharp for all the parameters involved and a characterization of the extremal sequences is given.

This and other similar results involving different subsequences are presented, including decompositions of sequences into  subsequences of bounded weight.
\end{abstract}

\section{Introduction}

A classical theorem in discrepancy theory (see \cite{Cha} as a general reference) is a celebrated theorem of Roth \cite{Ro}, which states that, for any positive integer $n$ and any function $f :[n] \rightarrow \{ -1 , 1\}$, there exists an arithmetic progression $A \subseteq [n]$ such that $| \sum_{x \in A} f(x)| \ge c n^{\frac{1}{4}}$ for some positive constant $c$. This bound is sharp up to a constant factor as shown later by Matousek and Spencer \cite{MaSp}. 
Another related result is Tao's \cite{Tao} recent proof of the Erd\H{o}s discrepancy conjecture, which states that, for any sequence $f:  \mathbb{N} \rightarrow \{-1,1\}$, the discrepancy $\sup_{n,d \in \mathbb{N}} |\sum_{j = 1}^n f(jd)|$ is infinite.

In this paper, we consider somewhat the opposite direction: instead of looking on the min-max problem which is the main concern of discrepancy theory, we consider the max-min problem. Specifically, in Theorem \ref{thm:k-blocks-t-sum} we prove the following:\\
Let $t$, $k$ and $q$ be integers such that $q \ge 0$, $0 \le t < k$, and $t \equiv k \,({\rm mod}\, 2)$, and let $s \in [0,t+1]$ be the unique integer satisfying $s \equiv q + \frac{k-t-2}{2} \,({\rm mod} \, (t+2))$. Then, for any integer $n$ such that
\[n \ge \max\left\{k,\frac{1}{2(t+2)}k^2 + \frac{q-s}{t+2}k - \frac{t}{2} + s\right\}\]
and any function $f:[n]\to \{-1,1\}$ with $|\sum_{i=1}^nf(i)| \le q$, there is a set $B \subseteq [n]$ of $k$ consecutive integers with $|\sum_{y\in B}f(y)| \le t$. This is best possible and the structure of the extremal sequences is also determined (see Theorem \ref{thm:sharpness_k-blocks-t-sum}). This result can be extended to cover the range $q = o(n)$,  and, in Theorem \ref{thm:o(n)}, we give an infinite version of this  phenomenon which we later  (in Section 5) apply for two specific examples with well-known number theoretic functions, the Liouville's function and  the Legendre symbol of quadratic  versus non quadratic residues.

As sequences of $k$ consecutive integers are, in fact, $k$-term arithmetic progressions with difference $1$, one could be tempted to ask weather Theorem \ref{thm:k-blocks-t-sum} offers the best possible value for arithmetic progressions, too. We show in Proposition \ref{prop:arithm-prog} that this is not always the case. The problem of finding the corresponding minimum positive integer for arithmetic progressions seems to be difficult. However, we leave the interested reader with a conjecture stating that the bound on $n$ should remain quadratic on $k$ (see Conjecture \ref{conj:arithm-prog}). Moreover, we consider  in Section 3 a relaxed version of $k$-term arithmetic progressions of common difference $d$, which we call $(d,k)$-blocks: sequences of $k$ integers $T = \{ a_1, \ldots ,a_k \}$  such that $a_1 < a_2 < \ldots < a_k$ and $a_{j+1} - a_j \le d$, for all $1 \le j \le k-1$.
We prove the following (Theorem \ref{thm:(d,k)blocks_k>4}):
\noindent
Let $k$, $d$ and $n$ be positive integers such that $k \ge 6$ is even, $d \ge 2$, and 
$n \ge \frac{d+1}{8}(k^2-2sk+4s-4) + 1$, where $s \in \{0,1\}$ with $s \equiv \frac{k-2}{2} \;({\rm mod}\; 2)$. Then, for any function $f:[n] \rightarrow \{-1,1\}$ such that $|\sum_{i=1}^nf(i)| \le \frac{(d-1)}{d+1}n$, there is a zero-sum $(d,k)$-block. Again we show that our bound is sharp and characterize the extremal sequences (see Theorem \ref{thm:sharpness_(d,k)blocks}). 

In Section  4, we turn our attention to the problem of decomposing a sequence $S = \{a_1, \ldots ,a_{nm} \}$, where $a_ j \in \{ -r, s\}$  for positive integers $r ,s$,  where $|\sum_{j=1}^{nm} a_ j| \le nq$, into $n$ $(2n-1, m)$-blocks $S_1, \ldots ,S_n$ with best possible upper bound on $\max_{1\le j \le n} |f(S_j)|$. This is done via a graph-theoretical approach. We mention that this problem is related to some results of Sevastyanov \cite{Sev}, Barany \cite{Bar}, and Ambrus-Barany-Grinberg \cite{AmBaGr} concerning rearranging bounded-sum sequences of vectors such that the partial sums are also bounded.  In the context of zero-sum subsequences over $\mathbb{Z}$, our results has some relations to those given for example in  \cite{AMShSiV}, \cite{CY} (Lemma 3.1), and \cite{SST}. 

%{\color{blue} In the next section,} we consider {\color{blue} functions} $f:[n] \to \{ -1 ,0 ,1 \}$ {\color{blue} versus functions with range $\{-1,1\}$ considered in the previous sections. Herefore, we} define  $h(k,q,t)$ to be the minimum integer  $m$, such that for every  $n\geq m$ and for every  $f :[n]\to \{ -1 ,1 \}$ with  $|f([n])|\leq q$, there is  a block   $B\subset[n]$ of $k$ consecutive integers  with  $|f(B)|\leq t$. Observe that for $k\not\equiv t$ (mod $2$)  and $t \geq 1$, $h(k,q,t) = h(k,q,t-1)$. {\color{blue} Further, we also define} $g(k,q,t)$ as the minimum integer  $m$, such that for every  $n\geq m$ and for every  $f :[n]\to \{ -1 ,0,1 \}$ with  $|f([n])|\leq q$, there is  a block   $B\subset[n]$ of $k$ consecutive integers   with  $|f(B)|\leq t$.  
%We prove the following: 
%1/ For  q >= 0 , t >= 1 and k = t (mod 2)  we have f(k,q,t) = g(k,q,t) .  
%2/  for q > = 1  t > = 1 and k ­ t(mod 2)  we have f(k,q,t) = f(k,q,t-1) = g(k,q,t).
%3/ for  q = 0 , t > = 1  and  k ­ t(mod 2)  we have   f(k,0,t) = f(k,0,t-1) =< g(k,0,t) =< f(k,1,t-1) .
%4/ for q = 0 , t = 0 and any k   : g(k,0,0)  doesnÕt exist .  

In the last section of this paper, we mention several open problems and we give also two simple number-theoretical applications when $|\sum_{1 \le j \le n}f(x)| = o(n)$, with a sketch of the proofs, where the full details can be easily completed by the interested reader. The first application deals with the Liouville function in a direction inspired by the work of Hildebrand \cite{Hil} on sign patterns of this function in short intervals. The second application is about zero-sum blocks of consecutive primes subject to the Legendre symbol of quadratic-non quadratic residues, which uses our result combined with the ``equidistribution'' of primes in arithmetic progressions.

\section{Existence of zero-sum blocks}

%\subsection{Definitions and notations}

We shall use the following definitions and notations.
Let $X$ be any set and let $Y$ be a subset of $X$. Given an integer function $f:X\to \Z$, we use $f(Y)$ to denote the sum of the function over all elements in $Y$, that is, $f(Y):=\sum_{y\in Y}f(y)$. Some times we called $f(Y)$ the  \emph{weight} of $Y$ with respect to $f$.  We say that $Y$ is a \emph{zero-sum set} with respect to $f$, if $f(Y)=0$.
 
As usual, $[n]$ denotes  the set  $\{1,2,...,n\}$. A \emph{k-block} is a set of $k$ consecutive integers.  A zero-sum $k$-block will be referred, for short, as a ZS $k$-block.

%a $1$-sum $k$-block

\begin{lemma}\label{lem:k-blocks-t-sum}
Let $t$, $k$ and $n$ be integers such that $t \equiv k \,({\rm mod}\, 2)$ and $|t| < k \leq n$, and let $f:[n]\to \{-1,1\}$. If there are $k$-blocks $S$ and $T$ in $[n]$ such that $f(S) < t$ and $f(T) >t$, then there is a $k$-block $B$ in $[n]$ with $f(B)= t$.
\end{lemma}

\begin{proof}
Denote the $n-k+1$ $k$-blocks in $[n]$ by $B_1,B_2, ... , B_{n-k+1}$, where $B_i:=\{i,...,i+k-1\}$ with $i \in \{1,...,n-k+1\}$. Suppose $f(B_i) \neq t$ for all $1 \le i \le n-k+1$. Assume first that $f(B_1)\geq t+1$. In fact, because of the parity of $k$ and $t$, $f(B_i) \equiv k \equiv t \,({\rm mod}\, 2)$, so we actually have $f(B_1)\geq t+2$. Consider the first $k$-block, say $B_l$, with $f(B_l) \le t-1$. With the same parity argument, we have $f(B_l) \le t-2$ but $f(B_{l-1}) \ge t+2$ and, therefore, $|f(B_{l-1})-f(B_l)|\geq 4$. But $B_{l-1}$ and $B_l$ differ in exactly two elements,  $l-1$ and $l+k-1$. Hence, $|f(B_{l-1})-f(B_l)|=|f(l-1)-f(l+k-1)|\in \{0,2\}$, a contradiction. The case $f(B_1)\leq t-1$ follows analogously.
\end{proof}

\begin{theorem}\label{thm:k-blocks-t-sum}
Let $t$, $k$ and $q$ be integers such that $q \ge 0$, $0 \le t < k$, and $t \equiv k \,({\rm mod}\, 2)$, and take $s \in [0,t+1]$ as the unique integer satisfying $s \equiv q + \frac{k-t-2}{2} \,({\rm mod} \, (t+2))$. Then, for any integer $n$ such that
\[n \ge \max\left\{k,\frac{1}{2(t+2)}k^2 + \frac{q-s}{t+2}k - \frac{t}{2} + s\right\}\]
and any function $f:[n]\to \{-1,1\}$ with $|f([n])| \le q$, there is a $k$-block $B \subseteq [n]$ with $|f(B)| \le t$.
\end{theorem}

\begin{proof}
Take $t$, $k$, $q$, $s$ and $n$ as in the statement of the theorem, and consider a function $f:[n]\to \{-1,1\}$ with $|f([n])| \le q$. We will prove the existence of a $k$-block $B \subseteq [n]$ with $|f(B)| \le t$.

If there are $k$-blocks $B$ and $D$ in $[n]$ such that $f(B)<t$ and $f(D) > t$, then we are done by Lemma \ref{lem:k-blocks-t-sum}. So either $f(B) \le t$ or $f(B) \ge t$ for all $k$-blocks $B$ in $[n]$. Analogously, if there are $k$-blocks $B$ and $D$ in $[n]$ such that $f(B)< -t$ and $f(D) > -t$, then we are done by Lemma \ref{lem:k-blocks-t-sum}. So either $f(B) \le -t$ or $f(B) \ge -t$ for all $k$-blocks $B$ in $[n]$. Altogether we obtain that either $f(B) \le -t$ for all $k$-blocks $B$, or $f(B) \ge t$ for all $k$-blocks $B$, or $-t \le f(B) \le t$ for all $k$-blocks $B$. In the third case there is nothing to prove. Moreover, if there is a $k$-block $B$ with $|f(B)| = t$, then we are done, too. So we may assume, without loss of generality, that $f(B) > t$ for all $k$-blocks $B$ (otherwise multiply everything by $-1$).  Since $t \equiv k \,({\rm mod}\, 2)$, note that we actually have $f(B) \ge t+2$ for all $k$-blocks $B$.

Now notice that
\begin{align*}
 n \ge \frac{1}{2(t+2)}k^2 + \frac{q-s}{t+2}k - \frac{t}{2} + s
 &= \frac{1}{2(t+2)}k^2 + \frac{q-s}{t+2}k + \frac{-t-2+t+2}{2(t+2)}k  - \frac{t}{2} + s\\
&=\frac{1}{2(t+2)}k^2 + \left(\frac{2q-2s-t-2}{2(t+2)}\right) k + \frac{1}{2}k - \frac{t}{2} + s\\
&=\left(\frac{k+2q-2s-t-2}{2(t+2)}\right) k + \frac{k - t}{2} + s,
\end{align*}
where, by hypothesis, $\frac{k+2q-2s-t-2}{2(t+2)}= \frac{q+ \frac{k-t-2}{2}-s}{t+2} $ is a non-negative integer, and $\frac{k-t}{2} + s$ is a positive integer. %$\frac{k+2q-2s-t-2}{2(t+2)} = \frac{q+ \frac{k-t-2}{2}-s}{t+2}  \ge 0$ and  $\frac{k-t}{2} + s  > 0$.

Assuming that $n=mk+r$, where $m$ and $r$ are integers such that $m\geq 0$ and $0 \leq r \leq k-1$, we have, by the above computation,
\begin{align}\label{eq:n=mk+r}
n = mk+r &\ge \left(\frac{k+2q-2s-t-2}{2(t+2)}\right) k + \frac{k-t}{2} + s,
\end{align}
and so $m \ge \frac{k+2q-2s-t-2}{2(t+2)}$.  Moreover, since  $n\geq k$, we actually have $m \ge \max\{1,\frac{k+2q-2s-t-2}{2(t+2)}\}$.
Now we split $[n]$ into $m$ disjoint consecutive $k$-blocks, $B_1, B_2, ..., B_m$, and a remainder block $R$ with $|R|=r$, so that $R=\emptyset$ if $r=0$ and $R=\{n-r+1,...,n\}$ otherwise. Since $f(B_i)\geq t+2$ for every $1\leq i\leq m$, we obtain
\begin{equation}\label{eq:first}
(t+2)m+f(R)\leq \sum_{i=1}^m f(B_i)+f(R)= f([n]) \le q,\end{equation} 
which yields
\begin{equation}\label{eq:wR}
f(R) \le q-(t+2)m.
\end{equation} 

Denote by $B^*$ the last $k$-block $\{n-k+1,...,n\}$, and let $T:=B^*\cap B_m$. Thus, $B^*=T\cup R$. By assumption, $f(B^*) = f(T)+f(R)\geq t+2$, which yields, together with (\ref{eq:wR}),
\begin{equation}\label{eq:wT}
f(T) \ge (t+2)(m+1) -q.
\end{equation} 

Now we distinguish two cases. Observe that, when $n = mk$, we have necessarily $m > \left(\frac{k+2q-2s-t-2}{2(t+2)}\right)$ because $\frac{k-t}{2} + s$ cannot be $0$. \\

\noindent
{\it Case 1: Suppose that $m =  \frac{k+2q-2s-t-2}{2(t+2)}$.} Recall that $m\neq 0$. Then, by (\ref{eq:n=mk+r}), it follows that $ 0  <  \frac{k-t}{2} + s \le r \le k-1$. On the other hand, inequality (\ref{eq:wT}) gives
\[f(T) \ge (t+2)(m+1)-q = \frac{k+t+2}{2} -s.\]
Using the fact that $s \le t+1$ and that $k > t$, we easily see that the right hand side of above inequality is positive. Hence, we obtain
\[|T| = \sum_{i \in T} |f(i)| \ge |f(T)| \ge \frac{k+t+2}{2} -s.\]
This in turn leads to $r  = |R| = |B^*| - |T| = k - |T| \le \frac{k-t}{2} + s - 1$, a contradiction to the lower bound for $r$ in this case.  

\noindent
{\it Case 2: Suppose that $m \ge \frac{k+2q-2s-t-2}{2(t+2)} + 1 = \frac{k+2q-2s+t+2}{2(t+2)}.$
}
Then, by (\ref{eq:wT}) and (\ref{eq:wR}), 
\[f(T) \ge (t+2)(m+1) -q  \ge \frac{k+2q-2s+3t+6}{2}-q = \frac{k+ 3t + 6}{2}-s, \mbox{ and}\]
\[f(R) \le q - (t+2) m \le q- \frac{k+2q-2s+t+2}{2} = s - \frac{t+k+2}{2}.\]
With $s \le t+1$, we easily note that the right hand side of the first inequality is positive, whereas, together with $t < k$, follows that the right hand side of the second one is negative. If $r=0$, we have already a contradiction. So it remains to handle the case $r \neq 0$. Then the above inequalities imply that 
\[|T|  = \sum_{i \in T} |f(i)| \ge |f(T)| \ge \frac{k+ 3t + 6}{2}-s, \mbox{ and}\]
\[|R|   = \sum_{i \in R} |f(i)| \ge |f(R)| \ge \frac{k+t+2}{2} - s,
\]
which, together, yield the contradiction $|B^*| = |T|+|R| \ge k +2t +4 -2s \ge k+2 > k$.\\
Since in both cases we obtain a contradiction, our assumption that all $k$ blocks $B$ have $f(B) > t$ is not possible and, thus, the proof is complete.
\end{proof}

Clearly, in a sequence of alternating $-1$'s and $1$'s of any length $n$ and for any $t \ge 2$ %(with $t \equiv k \,({\rm mod} \, 2)$)
, there is no
$k$-block $B$ with $f(B) = t$. However, the situation is different when $t \in \{ 0, 1\}$, as demonstrated by the two corollaries below.

\begin{corollary}\label{cor:zs}
Let $k \ge 2$ and $q \ge 0$ be integers with $k$ even, and take $s \in \{0,1\}$ as the unique integer satisfying $s \equiv q + \frac{k-2}{2} \,({\rm mod} \, 2)$. Then,  for any integer $n$ such that
\[n \ge \max\left\{k, \frac{k^2}{4} + \frac{q-s}{2}k+s\right\}\]
and any function $f:[n]\to \{-1,1\}$ with $|f([n])| \le q$, there is a ZS $k$-block in $[n]$.
\end{corollary}

\begin{proof}
Setting $t=0$ in Theorem \ref{thm:k-blocks-t-sum}, the statement follows easily.
\end{proof}

A $k$-block $B$ such that $|f(B)|=1$ is called  a \emph{$1$-sum $k$-block}.
 
\begin{corollary}\label{cor:1s}
Let $k \ge 3$ and $q \ge 0$ be integers with $k$ odd, and take $s \in \{0,1, 2\}$ as the unique integer satisfying $s \equiv q + \frac{k-3}{2} \,({\rm mod} \, 3)$. Then, for any integer $n$ such that
\[n \ge \max\left\{k,\frac{k^2}{6} + \frac{q-s}{3} k + s - \frac{1}{2}\right\}\]
and any function $f:[n]\to \{-1,1\}$ with $|f([n])| \le q$, there is a $1$-sum $k$-block in $[n]$.
\end{corollary}

\begin{proof}
Setting $t=1$, and noting that, by parity reasons, $|f(B)| \neq 0$ for any $k$-block $B$ of $[n]$, the statement follows easily by Theorem \ref{thm:k-blocks-t-sum}.
\end{proof}

Clearly, if $\max\left\{k,\frac{1}{2(t+2)}k^2 + \frac{q-s}{t+2}k - \frac{t}{2} + s\right\}=k$ Theorem \ref{thm:k-blocks-t-sum} is tight, since with one unit below we do not even have $k$-blocks. Theorem \ref{thm:sharpness_k-blocks-t-sum}  below shows that Theorem \ref{thm:k-blocks-t-sum} is tight  also when $\frac{1}{2(t+2)}k^2 + \frac{q-s}{t+2}k - \frac{t}{2} + s>k$, meaning that, for $n=\frac{1}{2(t+2)}k^2 + \frac{q-s}{t+2}k - \frac{t}{2} + s-1$, there are examples of functions having the highest possible value for $f([n])$, namely $q$, such that no $k$-block $B \subseteq [n]$ has weight $|f(B)| \le t$. With this aim we define, for  suitable integers $k$, $t$ and $q$, the family of functions $\mathcal{F}_{k,t,q}$. Let $k$, $t$ and $q$ be integers satisfying that $q \ge 0$, $0 \le t < k$, $t \equiv k \,({\rm mod}\, 2)$,  and $\frac{1}{2(t+2)}k^2 + \frac{q-s}{t+2}k - \frac{t}{2} + s> k$, where $s \in [0,t+1]$ with $s \equiv q + \frac{k-t-2}{2} \,({\rm mod} \, (t+2))$. Let $n=\frac{1}{2(t+2)}k^2 + \frac{q-s}{t+2}k - \frac{t}{2} + s-1$, and write $n=km+r$ where $m$ and $r$ are integers such that $m \ge 0$ and $0 \leq r \le k-1$. Note that, by assumption, we actually have  $m \ge 1$. Then we define a function $f$ to belong to  $\mathcal{F}_{k,t,q}$ if and only if $f:[n] \rightarrow \{-1,1\}$ is such that
\[\pm f(x) = \left\{
\begin{array}{rl}
-1, & x \in R \setminus S,\\
1, & \mbox{otherwise},
\end{array}
\right.
\] 
for subsets $R$ and $S$ of $[n]$ with $R=S=\emptyset$ if $r=0$ and $R = \bigcup_{i=1}^{m+1} R_i$, $S = \bigcup_{i=1}^{m+1} S_i$ otherwise, where
\[R_i= \{(i-1)k+1,(i-1)k+2, \ldots, (i-1)k+r\},\] 
\[S=\emptyset, \mbox{ if } s=0, \mbox{ and,}\]
\[\mbox{if } s \neq 0, \; S_i = \{a_1^i, a_2^i, \ldots, a_s^i\} \subseteq R_i \] 
is such that $a_1^i < a_2^i < \ldots < a_s^i$ and $a^{i+1}_{\ell} \le a^{i}_{\ell} + k$, for $i \in [m]$ and $\ell \in [s]$.\\

%It is important to note that, by definition, if $s=0$ then there is a unique function $f_{k,t,q}$, while, if $s\neq 0$, there are several functions $f_{k,t,q}$ depending on the sets $S_i$'s}.

The family of functions $\mathcal{F}_{k,t,q}$ will prove to be precisely   the family of functions for which Theorem \ref{thm:k-blocks-t-sum} does not work anymore as soon as the $n$ is chosen one unit below the given lower bound. For the sake of comprehension, we shall see, before continuing, some particular examples of $\mathcal{F}_{k,t,q}$ for different values of $k$, $t$ and $q$. In Example~\ref{ex:ZS}, we exhibit functions corresponding to the tightness of Corollary \ref{cor:zs}, that is, when $t=0$: we fix $k=6$ and consider $q\in \{0,1\}$.  In Example~\ref{ex:1S}, we exhibit a function corresponding to the tightness of Corollary \ref{cor:1s},  that is, when $t=1$, where, in order to illustrate how the construction works for a larger $q$, we consider the case $k=7$, $q=4$. 

\begin{example}\label{ex:ZS}
Fix $k=6$ and $t=0$. We analyze two cases corresponding to $q\in\{0,1\}$. According to the hypothesis of Corollary \ref{cor:zs}, for $q=0$  we choose $s=0$ and, for $q=1$ we choose $s=1$. In both cases we take $n=\frac{k^2}{4} + \frac{q-s}{2}k+s-1$. 
\begin{itemize}
\item[(a)] Let $k=6$, $t=0$ and $q=0$. Thus, $s=0$ and $n=8$. Hence, $m=1$ and $r=2$. By definition,  $S=\emptyset$,  and $R=R_1\cup R_2$ where $R_1=\{1,2\}$ and $R_2=\{7,8\}$. Thus  $\mathcal{F}_{6,0,0}=\{f,-f\}$ where $f$ is the function represented by: \[-\,-\,+\,+\,+\,+\,-\,-,\]
where the $+$'s and $-$'s represent the values $+1$ or, respectively, $-1$  of $f$ at the corresponding position.

\item[(b)] Let $k=6$, $t=0$ and $q=1$. Thus, $s=1$, $n=9$, and, hence, $m=1$ and $r=3$. By definition,  $R=R_1\cup R_2$ where $R_1=\{1,2,3\}$ and $R_2=\{7,8,9\}$, and $S=S_1\cup S_2$ with $S_i=\{a_1^i \}\subseteq R_i$, for $i\in\{1,2\}$, such that $a_1^2\leq a_1^1+6$.  These gives the following family  $\mathcal{F}_{6,0,1}=\{\pm g_i| 1 \le i \le 6\}$, where:
\begin{center}
$g_1: \;\;\; +\,-\,-\,+\,+\,+\,+\,-\,-$

$g_2: \;\;\; -\,+\,-\,+\,+\,+\,+\,-\,-$

$g_3: \;\;\; -\,+\,-\,+\,+\,+\,-\,+\,-$

$g_4: \;\;\; -\,-\,+\,+\,+\,+\,+\,-\,-$

$g_5: \;\;\; -\,-\,+\,+\,+\,+\,-\,+\,-$

$\mbox{ }\, g_6: \;\;\; -\,-\,+\,+\,+\,+\,-\,-\,+.$
\end{center}
\end{itemize}

Observe that the functions $f$ and $-f$ in case (a) fulfill $|f([8])| = |-f([8])| = 0$ and have no ZS $6$-blocks, while  all functions $\pm g_i$ in case (b) satisfy $|g_i([9])| =|-g_i([9])| = 1$, for $1 \le i \le 6$, and do neither have ZS $6$-blocks. 
\end{example}

\begin{example}\label{ex:1S}
Let  $k=7$, $t=1$ and $q=4$. According to the hypothesis of Corollary \ref{cor:1s}, we choose $s=0$ and $n=\frac{k^2}{6} + \frac{q-s}{3} k + s  - \frac{1}{2}-1 = 16$, which gives $m=2$ and $r=2$. By definition,  $S=\emptyset$ and $R=R_1\cup R_2\cup R_3$, where $R_1=\{1,2\}$, $R_2=\{8,9\}$, and $R_3=\{15,16\}$. Thus  $\mathcal{F}_{7,1,4}=\{h,-h\}$ where $h$ is represented by: 
\[-\,-\,+\,+\,+\,+\,+\,-\,-\,+\,+\,+\,+\,+\,-\,-.\]
Observe that the functions $h$ and $-h$ fulfill $|h([16])| =|-h([16])| =  4$ and have no  $1$-sum $7$-blocks.
\end{example}

\begin{theorem}\label{thm:sharpness_k-blocks-t-sum}
Let $k$, $t$ and $q$ be integers such that $q \ge 0$, $0 \le t < k$,  $t \equiv k \,({\rm mod}\, 2)$, and $\frac{1}{2(t+2)}k^2 + \frac{q-s}{t+2}k - \frac{t}{2} + s>k$, where
$s \in [0,t+1]$ is the unique integer satisfying $s \equiv q + \frac{k-t-2}{2} \,({\rm mod} \, (t+2))$. Then, for
\[n = \frac{1}{2(t+2)}k^2 + \frac{q-s}{t+2}k - \frac{t}{2} + s - 1,\]
a function $f :[n] \rightarrow \{-1,1\}$ satisfies $|f([n])| = q$ and $|f(B)| > t$ for all $k$-blocks $B \subseteq [n]$ if and only if $f \in \mathcal{F}_{k,t,q}$.
\end{theorem}

\begin{proof}
By assumption we have  $n \ge k$. First of all, we will show that, for any $f \in \mathcal{F}_{k,t,q}$, $|f([n])| = q$ and $|f(B)| \ge t+2$ hold for all $k$-blocks $B \subseteq [n]$. So let $f \in \mathcal{F}_{k,t,q}$ be such that $f(x) = -1$ if $x \in R \setminus S$ and $f(x) = 1$ else for adequate subsets $R, S \subseteq [n]$ (the case of function $-f$ is analogous). Similar to previous theorem, we write $n$ as follows:
\[n = \left(\frac{k+2q-2s-t-2}{2(t+2)}\right) k + \frac{k - t}{2} + s -1.\]
Let $m = \frac{k+2q-2s-t-2}{2(t+2)}$ and $r = \frac{k - t}{2} + s-1$. 
Note that, since  $t \le k-2$ and $s-1\leq t$, we have

 \[r = \frac{k - t}{2} + s-1 \ge \frac{k - (k-2)}{2} -1 = 0, \mbox{ and} \]
 \[r =  k -\frac{k + t}{2} + s - 1 \le k -\frac{(t+2)+t}{2} + t  = k-1.\]
Hence, $n = mk+r$ represents the division of $n$ by $k$ with remainder $r$. Observe that $n \geq k$ implies that $m \ge 1$. 
Now, by the definition of $f \in \mathcal{F}_{k,t,q}$, we obtain
\begin{align*}
f([n]) &= n - 2|R \setminus S|\\
& = (mk+r) - 2r(m+1) + 2 s(m+1)\\
& = m(k-2r+2s) - r + 2s\\
& = m \left(k-2 \left(\frac{k - t}{2} + s-1\right)+2s\right) - \left(\frac{k - t}{2} + s-1\right) + 2s\\
& = m(t+2) + \frac{t-k}{2}+s+1\\
& = \frac{k+2q-2s-t-2}{2} + \frac{t-k}{2}+s+1 = q,
\end{align*}
as claimed.\\

Now, we split $[n]$ into $m$ consecutive $k$-blocks $B_1, B_2, \ldots, B_m$, and a remainder block $R$, where $B_i = \{(i-1)k+1, (i-1)k+2, \ldots, ik\}$, and $R = \emptyset$ in case $r=0$ and $R=\{mk+1,...mk+r\}$ otherwise. Note that $R=R_{m+1}$, and, if $r \neq 0$, the first $r$ elements of each $B_i$ are precisely the elements of $R_i = \{(i-1)k+1,(i-1)k+2, \ldots, (i-1)k+r\}$ (as defined before the theorem), and so $T_i := B_i \setminus R_i$ is the block containing the remaining $k-r$ elements of $B_i$. Hence, we can write 
\begin{equation}\label{eqn:partition}
[n] = \left(\bigcup_{i=1}^{m} B_i\right) \cup R_{m+1} = \left(\bigcup_{i=1}^{m} (R_i  \cup T_i) \right) \cup R_{m+1}.
\end{equation}
Let $B \subseteq [n]$ be any $k$-block. We will show that $|f(B)| \ge t+2$. If $r = 0$, then $f(x) = 1$ for all $x \in  [n]$ and, hence, $f(B) \ge k \ge t+2$ and we are done. Thus we may assume that $r > 0$. Then either there is a $j \in [m]$ such that $B \subseteq R_j \cup T_j \cup R_{j+1}$, or (provided $m \ge 2$) there is a $j \in [m-1]$ such that $B \subseteq T_j \cup R_{j+1} \cup T_{j+1}$. Now define
\[
U = \{u \in B \cap (R_j \cup R_{j+1}) \;|\; f(u) = 1\}
\]
and recall that $S_i = \{a_1^i, a_2^i, \ldots, a_s^i\}$ is the set of elements of $R_i$ having $f$-value equal to $1$, for each $i \in [m+1]$. In order to prove that $|f(B)| \ge t+2$, we will show first that $|U| \ge s$. So we suppose to the contrary that $|U| < s$. We distinguish two cases.\\

\noindent
\emph{Case 1: $U \cap R_j \neq \emptyset$ and $U \cap R_{j+1} \neq \emptyset$.}\\
Let $y \in [s]$ be the first index with $a_{y}^j \in U \cap R_j$ and let $z \in [s]$ be the last index with $a_{z}^j \in U \cap R_{j+1}$. Then $s > |U| = |U \cap R_j| + |U \cap R_{j+1}| = (s-y+1) + z$ and, thus, $z + 1 < y$. Moreover, since $|U| < s$, neither $S_j$ nor $S_{j+1}$ are fully contained in $U$. Thus, $y > 1$ and $z < s$, and $a_{y-1}^j, a_{z+1}^{j+1} \notin B$. Since $z + 1 < y$ and the block $B$ lies in between the two positions $a_{y-1}^j$ and $a_{z+1}^{j+1}$, we obtain 
\[a_{y-1}^j +k < a_{z+1}^{j+1} \le a_{y-1}^{j+1},\]
a contradiction to the definition of $S_j$ and $S_{j+1}$.\\

\noindent
\emph{Case 2: $U \cap R_j \neq \emptyset$ and $U \cap R_{j+1} = \emptyset$, or $U \cap R_j = \emptyset$ and $U \cap R_{j+1} \neq \emptyset$.} \\
By symmetry reasons, assume, without loss of generality, that $U \cap R_j \neq \emptyset$ and $U \cap R_{j+1} = \emptyset$. Then $a_1^{j+1} \notin U$ and, since $|U| < s$, $a_1^j \notin U$. This implies that $a_1^{j}$ and $a_1^{j+1}$ lie one on the left side and one on the right side of $B$, leading to the contradiction $a_1^j + k < a_1^{j+1}$.\\

Hence, we have shown that $|U| \ge s$. Recall that  $f(T_j)=|T_j|$ by definition. In case that $B \subseteq R_j \cup T_j \cup R_{j+1}$, using the fact that $T_j \subseteq B$ and $|T_j| = |B_j \setminus R_j| = k - |R_j| = k - r = k - \left(\frac{k - t}{2} + s-1\right) =\frac{k + t}{2} - s+1$, we obtain
\begin{align*}
f(B) &= f(B \cap (R_j \cup R_{j+1})) + f(B \cap T_j) \\
&=f(B \setminus T_j) + |T_j|\\
&= 2|U| - |B \setminus T_j| + |T_j|\\
&= 2|U| - k + 2|T_j|\\
&\ge 2s - k + 2 \left(\frac{k + t}{2} - s+1\right)\\
&= t+2.
\end{align*}
On the other side, if $B \subseteq T_j \cup R_{j+1} \cup T_{j+1}$, then $R_{j+1} \subseteq B$ and, with $|R_{j+1}| = r = \frac{k - t}{2} + s-1$, it follows that
\begin{align*}
f(B) &= f(B \cap (T_j \cup T_{j+1})) + f(B \cap R_{j+1}) \\
&=f(B \setminus R_{j+1}) + f(R_{j+1})\\
&= |B \setminus R_{j+1}| - |R_{j+1}| + 2|U|\\
&= k - 2|R_{j+1}| + 2|U|\\
&\ge k - 2 \left(\frac{k - t}{2} + s-1\right) + 2s\\
&= t+2.
\end{align*}

Hence, we have proved that $f(B) \ge t+2$ for every $k$-block $B \subseteq [n]$. \\

Finally, we will show that every function $f:[n] \rightarrow \{-1,1\}$  satisfying  $|f([n])| = q$ and $|f(B)| \ge t+2$ for all $k$-blocks $B \subseteq [n]$ is contained in $\mathcal{F}_{k,t,q}$. So let $f$ be such a function. Assume first that $f([n]) \ge 0$. By Lemma \ref{lem:k-blocks-t-sum}, we can assume that $f(B) \ge t+2$ for every $k$-block $B \subseteq [n]$. Consider first the case $r = \frac{k-t}{2}+s-1 = 0$. Then we have $s = \frac{t+2-k}{2}$ and, hence, $m = \frac{q+k-t-2}{t+2}$. Moreover, if we split the set $[n]$ into $m$ $k$-blocks $B_1, B_2, \ldots, B_m$, the conditions $f(B_i) \ge t+2$, for $1 \le i \le m$, and $k \ge t+2$ give 
\[q \le f([n]) = \sum_{i=1}^m f(B_i) \le m(t+2) = q+k-t-2 \le q,\]
implying that $f(B_i) = t+2 = k$ for all $1 \le i \le m$. Hence, $f(x)=1$ for all $x \in [n]$ and, thus, $S = \emptyset$ and we are done. Therefore, we can assume from now on that $r > 0$.

We consider again the partition given in (\ref{eqn:partition}) and let $C_i = T_i \cup R_{i+1}$, for each $i \in [m]$. Then $C_i$ is a $k$-block for each $i \in [m]$ and, we can write, for any $j \in [m+1]$,
\[
[n] = \left(\bigcup_{i=1}^{j-1} B_i\right) \cup R_j \cup \left(\bigcup_{i=j}^{m} C_i\right).
\]
Further, we have
\begin{equation*}%\label{eqn:B_i-R_j-C_i)}
q = f([n]) = \left(\sum_{i=1}^{j-1} f(B_i) \right) + f(R_j)+ \left(\sum_{i=j}^{m} f(C_i) \right) \ge m(t+2) + f(R_j). 
\end{equation*}
Since $m(t+2) = \frac{k+2q-2s-t-2}{2} = q-\frac{t-k}{2}-s-1$, it follows that
\begin{equation*}%\label{eqn:\hat{f}(R_j)}
f(R_j)\le \frac{t-k}{2}+s+1.
\end{equation*}
Thus, as $|T_j| =k-r = \frac{k + t}{2} - s+1$, we obtain
\begin{align*}
t+2 &\le f(B_j) = f(R_j) + f(T_j)\\
&\le f(R_j) + |T_j|\\
&\le  \left(\frac{t-k}{2}+s+1\right) +  \left(\frac{k + t}{2} - s+1\right) = t+2,
\end{align*}
Analogously, 
\begin{align*}
t+2 &\le f(C_j) =  f(T_j) + f(R_{j+1})\\
&\le |T_j| + f(R_{j+1})\\
&\le  \left(\frac{k + t}{2} - s+1\right) + \left(\frac{t-k}{2}+s+1\right)= t+2.
\end{align*}
Since $j$ was taken arbitrarily, it follows that $f(T_i) = |T_i| = \frac{k + t}{2} - s+1$ for every $i \in [m]$, and $f(R_i) = \frac{t-k}{2}+s+1$ for every $i \in [m+1]$. Because of $f(T_i) = |T_i|$, it follows that $f(x) = 1$ for every $x \in T_i$ and $i \in [m]$. It remains to show that the $f$-values on the $R_i$'s follow one of the patterns given in the definition of $\mathcal{F}_{k,t,q}$. With this aim, let $\alpha$ and $\beta$ be the number of $f$-values of $R_i$ which are equal to $-1$ and $1$, respectively. Then we have
\[f(R_i) = \beta - \alpha = \frac{t-k}{2}+s+1, \mbox{ and } \]
\[|R_i| = \alpha + \beta = \frac{k-t}{2}+s-1.\]
Solving this equality system, we obtain $\alpha = \frac{k-t}{2}-1$ and $\beta = s$. Let $S_i = \{{a}_1^i, {a}_2^i, \ldots, {a}_s^i\}$ be the set of elements of $R_i$ having ${f}$-value equal to $1$, where ${a}_1^i < {a}_2^i< \ldots < {a}_s^i$. We need to show that the condition ${a}^{i+1}_{\ell} \le {a}^i_{\ell} + k$ is satisfied for every $i \in [m]$ and $\ell \in [s]$. So assume for contradiction that ${a}^{i+1}_{\ell} > {a}^i_{\ell} + k$ for some $i \in [m]$ and $\ell \in [s]$. Let ${B}$ be the $k$-block $\{{a}_{\ell}^i+1, {a}_{\ell}^i+2, \ldots, {a}_{\ell}^i+k\}$. Then $T_i \subseteq {B} \subseteq R_i \cup T_i \cup R_{i+1}$. Moreover, ${B} \cap {S}_i = \{{a}_{\ell+1}^i, {a}_{\ell+2}^i, \ldots, {a}_s^i\}$ and, since ${a}^{i+1}_{\ell} > {a}^i_{\ell} + k$, ${B} \cap {S}_{i+1} \subseteq \{{a}_1^{i+1}, {a}_2^{i+1}, \ldots, {a}_{\ell-1}^{i+1}\}$. This implies
\begin{align*}
{f}({B}) &= {f}({B} \cap R_i) + f(T_i) + {f}({B} \cap R_{i+1})\\
&= 2|{B} \cap {S}_i|-|{B} \cap R_i| + |T_i| + 2|{B} \cap {S}_{i+1}| - |{B} \cap R_{i+1}|\\
&= 2|{B} \cap {S}_i|+ |T_i| + 2|{B} \cap {S}_{i+1}| - |{B} \setminus T_i|\\
&\le 2(s-\ell) + \left(\frac{k + t}{2} - s+1\right) + 2(\ell-1) - \left( \frac{k-t}{2}+s-1\right)\\
&=t,
\end{align*}
a contradiction to the assumption that every $k$-block has weight at least $t+2$. Hence, we have shown that the condition ${a}^{i+1}_{\ell} \le {a}^i_{\ell} + k$ is satisfied for every $i \in [m]$ and ${\ell} \in [s]$, and, therefore ${f} \in \mathcal{F}_{k,q,t}$, with ${f}(x) = -1$, for $x \in R \setminus S$, and $f(x) = 1$ else, where $S = \bigcup_{i=1}^{m+1} {S}_i$. Finally, if ${f}([n]) \le 0$, it follows analogously (multiplying everything by ($-1$)) that ${f} \in \mathcal{F}_{k,q,t}$.
\end{proof}

To conclude this section we point out that Theorem \ref{thm:k-blocks-t-sum} can be extended to cover the range $q=o(n)$ and we present an infinite version of it which is a direct consequence of Lemma \ref{lem:k-blocks-t-sum}.

\begin{theorem}\label{thm:o(n)}
Let $f:\mathbb{Z}^+\to \{-1,1\}$ be a function such that $|f([n])|\leq \frac{n}{\omega (n)}$ where $\omega(n)\to \infty$ as $n \to \infty$. Then, for every even $k\geq 2$, there are infinitely many ZS $k$-blocks.
\end{theorem}

\begin{proof}
Fix $n_0\equiv 1$ (mod $k$) and consider $n > (k+2)(n_0-1)$ and such that $n \equiv 0$ (mod $k$), and $\omega(m)\geq k$ for every $m\geq n_0$. Consider the partition of the interval $[n_0,n]$ into $\frac{n-n_0+1}{k}$ disjoint $k$-blocks, $B_1$, $B_2$, ... ,$B_{\frac{n-n_0+1}{k}}$.
Suppose that $f(B_i)\neq 0$ for all $i\in\{1,...,\frac{n-n_0+1}{k}\}$. We know, by taking $t=0$ in Lemma \ref{lem:k-blocks-t-sum}, that if $f(B_i)<0<f(B_j)$ for some $i,j\in\{1,...,\frac{n-n_0+1}{k}\}$ then there will be a ZS $k$-block. So we may assume, without lost of generality, that all blocks $B_i$ satisfy $f(B_i)\geq 1$, and hence, by parity, $f(B_i)\geq 2$. Then 
\[ 2 \left(\frac{n-n_0+1}{k}\right)\le \sum_{i=1}^{\frac{n-n_0+1}{k}}f(B_i)  \le \left| \sum_{j=n_0}^{n}f(j)\right| \le |f([n])|+n_0 -1 \le \frac{n}{\omega (n)}+n_0 -1 \le  \frac{n}{k}+n_0 -1, 
 \]
from which it follows that,
\[
\frac{n}{k} \le n_0-1+\frac{2(n_0-1)}{k}=(k+2)\frac{(n_0-1)}{k},
\]
a contradiction to $n>(k+2)(n_0-1)$.

Since $n_0$ is fixed but can be arbitrarily large, we can build a sequence $\{n_0, n_1,n_2,...\}$ such that, in each of the intervals $[n_j,n_{(j+1)}]$, $j\geq 0$, there is a ZS $k$-block. Hence, there are infinitely many ZS $k$-blocks.
\end{proof}

\begin{remark}
Observe that the domains $[n]$ or $\mathbb{Z}^+$ of the theorems of this section can be replaced by any finite sequence $A = \{a_1, a_2, \ldots, a_n\}$ or, respectively, an infinite sequence $A = (a_i)_{i \ge 1}$, where a $k$-block is a set of $k$ consecutive terms in the sequence.
\end{remark}

\section{Existence of zero-sum bounded gap subsequences}

Let $d$ and $k$ be two positive integers and let $A = \{a_1, a_2, \ldots, a_k\}$ be a sequence of pairwise different integers in increasing order. We will call $A$ a \emph{$d$-bounded gap sequence} if $a_{i+1} - a_i \le d$ for all $1 \le i \le k-1$. In particular, the case $d = 1$ corresponds to a $k$-block. A zero-sum $d$-bounded gap sequence $A$ of length $k$ will be denoted \emph{ZS $(d,k)$-block} and so ZS $(1,k)$-blocks are just ZS $k$-blocks. 

\begin{theorem}\label{thm:(d,k)blocks}
Let $k$, $d$ and $n$ be positive integers such that $k$ is even, $d \ge 2$, and $n \ge \max\{k, \frac{d+1}{8}(k^2-2sk+4s-4) + 1\}$, where $s \in \{0,1\}$ with $s \equiv \frac{k-2}{2} \;({\rm mod}\; 2)$. Then, for any function $f:[n] \rightarrow \{-1,1\}$ such that $|f([n])| \le  \min \{n-k,\frac{d-1}{d+1}n\}$, there is a ZS $(d,k)$-block.
\end{theorem}

\begin{proof}
 Let $f:[n] \rightarrow \{-1,1\}$ be a function such that $|f([n])| \le  \min \{n-k,\frac{d-1}{d+1}n\}$. Let $F^+ = f^-(1)$ and $F^-=f^{-1}(-1)$ and suppose, without loss of generality, that $|F^+| \ge |F^-|$. Then $|f([n])|= f([n]) = |F^+| - |F^-| = n - 2|F^-|$ and, together with $|f([n])| \le \frac{(d-1)}{d+1}n$, we obtain
\begin{equation}\label{F-}
|F^-| = \frac{1}{2}(n - f([n])) \ge \frac{1}{d+1} n
\end{equation}
\begin{equation}\label{F+}
|F^+| = n - |F^-| \le \frac{d}{d+1}n.
\end{equation}
For $1 \le i \le d$, let $S_i = \{ x \in [n] \;|\; x \equiv i \;({\rm mod}\; d)\}$. Choose $S = S_i$ such that $|S_i \cap F^+|$ is minimum. Then $|S \cap F^+| \le \frac{1}{d}|F^+|$ and, thus, in $[n]\setminus S$ there are at least $|F^+| - \frac{1}{d}|F^+| = \frac{d-1}{d}|F^+|$ positions with $f$-value $1$. Now, with inequalities (\ref{F-}) and (\ref{F+}), we can deduce that there are at least
\[\frac{d-1}{d} |F^+| = |F^+| - \frac{1}{d}|F^+|  \ge |F^+| - \frac{1}{d+1}n \ge |F^+| - |F^-| = f([n])\]
positions in $[n] \setminus S$ with $f$-value $1$. Hence, we can delete the first $f([n])$ $1$'s appearing in the sequence $f(1), f(2), \ldots, f(n)$ which are not on positions contained in $S$ and we obtain a new sequence of $1$'s and $-1$'s with $n - f([n]) \ge n - \frac{(d-1)}{d+1}n = \frac{2}{d+1} n$ terms and total sum equal to $0$. Since by hypothesis $n \ge  \frac{d+1}{8}(k^2-2sk+4s-4) + 1$, where $s \in \{0,1\}$ with $s \equiv \frac{k-2}{2} \;({\rm mod}\; 2)$, and $|f([n])| \le  \min \{n-k,\frac{d-1}{d+1}n\}$, it follows that  $n - f([n])\ge \max \{k, \frac{2}{d+1} n \}$, implying that
\begin{align*}
n - f([n]) &\ge \max \left\{k, \left\lceil \frac{2}{d+1} n\right\rceil \right\} \\
&= \max \left\{k, \left\lceil\frac{1}{4}(k^2-2sk+4s-4) + \frac{2}{d+1}\right\rceil \right\} \\
&= \max \left\{k, \frac{1}{4}(k^2-2sk+4s) \right\},
\end{align*}

satisfying the hypothesis of Theorem \ref{thm:k-blocks-t-sum} with $t = q = 0$. Thus, in the new sequence, we can guarantee the existence of a ZS $k$-block $B$. This set $B$ is, by construction, a ZS $(d,k)$-block in the original sequence.
\end{proof}

Observe that, if $k \ge 6$, the assumptions $n \ge \frac{d+1}{8}(k^2-2sk+4s-4) + 1$ and $|f([n])| \le \frac{d-1}{d+1}n$ imply already that $n \ge k$ and $|f([n])| \le n-k$. This is because of the following computations. Assuming $k \ge 6$, notice first that, if either $s=0$ or $s=1$, $k^2-2sk+4s-4 \ge k^2-2k = k(k-2)$. Hence, with $d \ge2$, we have, on the one hand,
\[\frac{d+1}{8}(k^2-2sk+4s-4)+1 \ge \frac{d+1}{8} k (k-2)+1 \ge \frac{d+1}{2}k > k.\] 
On the other hand, we have
\[n \ge \frac{d+1}{8}(k^2-2sk+4s-4) + 1 \ge \frac{d+1}{8}k(k-2) + 1 \ge \frac{d+1}{2}k + 1 > \frac{d+1}{2}k,\]
which implies that $n-k > n - \frac{2n}{d+1} = \frac{d-1}{d+1}n$.

Therefore, Theorem \ref{thm:(d,k)blocks} can be stated the following way for $k \ge 6$.

\begin{theorem}\label{thm:(d,k)blocks_k>4}
Let $k$, $d$ and $n$ be positive integers such that $k \ge 6$ is even, $d \ge 2$, and $n \ge\frac{d+1}{8}(k^2-2sk+4s-4) + 1$, where $s \in \{0,1\}$ with $s \equiv \frac{k-2}{2} \;({\rm mod}\; 2)$. Then, for any function $f:[n] \rightarrow \{-1,1\}$ such that $|f([n])| \le  \frac{d-1}{d+1}n$, there exists a ZS $(d,k)$-block.
\end{theorem}

For  $k = 2$ and $k = 4$, it follows that, as soon as the natural condition $n \ge k$ is fulfilled, Theorem \ref{thm:(d,k)blocks} guarantees the existence of a ZS $(d,k)$-block for any function $f:[n] \rightarrow \{-1,1\}$ such that $|f([n])| \le  \min\{n-k, \frac{d-1}{d+1}n\}$. In case that $k \ge 6$, we will show in the next theorem that Theorem \ref{thm:(d,k)blocks_k>4} is sharp, meaning that, as soon as the $n$ is chosen one unit below of the given lower bound, there are examples of functions having $|f([n])|  = \frac{d-1}{d+1}n$, such that no $(d,k)$-block $B \subseteq [n]$ is a zero-sum set. Moreover, we will characterize the extremal sequences. With this aim, we will define, for given positive integers $d$ and $k$, the family of functions $\mathcal{F}_{d,k}$. 

Let $d$ and $k$ be positive integers, such that $k>4$ is even and $d\ge 2$, and consider $n = \frac{d+1}{8}(k^2-2sk+4s-4)$ where $s\in\{0,1\}$ with $s \equiv \frac{k-2}{2} \;({\rm mod}\; 2)$. First, note that 
\begin{align*}
 \frac{d+1}{8}\left(k^2-2sk+4s-4\right) &= \frac{d+1}{8}\left(k^2-2sk-2k+2k+4s-4\right)\\
 &=\frac{d+1}{8}\left(k^2-2sk-2k\right)+\frac{d(k+2s-2)+k+2s-2}{4}\\
 &=\frac{d+1}{8}\left(k^2-2sk-2k\right)+\frac{d(k-2s-2)+4ds+k+2s-2}{4}\\
 %&=...\\
 &=\frac{d+1}{8}k(k-2s-2)+\frac{d(k-2s-2)-(k-2s-2)+2k-4+4ds}{4}\\
 &=\frac{d+1}{8}k(k-2s-2)+\frac{(d-1)(k-2s-2)}{4}+\frac{k}{2}-1+ds\\
&=\frac{k-2s-2}{4}\left( (d+1)\frac{k}{2}+d-1 \right)+\frac{k}{2}-1+ds.
\end{align*}
Thus, $n=\frac{d+1}{8}(k^2-2sk+4s-4)=mb+r$, where 

\begin{equation}\label{eq:mbr}
m=\frac{k-2s-2}{4}, \, \mbox{ } \, b=(d+1)\frac{k}{2}+d-1\,  \mbox{ and }  \, r=\frac{k}{2}-1+ds. 
\end{equation}

Note that $k \ge 6$ implies that $m>0$. We will decompose the interval $[n]$ into $m+1$ blocks of length $r$, and $m$ blocks of length $b-r$. Let
\[R_i=\left\{(i-1)b+1, \dots ,(i-1)b + r\right\} \mbox{ for } 1\le i\le m+1 \mbox{, and }\]
\[T_i=\left\{(i-1)b +\frac{k}{2}+ds, \dots ,ib\right\} \mbox{ for } 1\le i\le m.\]
Hence,
\[[n]=\left(\bigcup_{i=1}^m(R_i\cup T_i)\right)\cup R_{m+1}.\]

Now we define the family of functions $\mathcal{F}^+_{d,k}$ as follows. We say that $f:[n] \rightarrow \{-1,1\}$ is contained in $\mathcal{F}^+_{d,k}$ if and only if $f$ satisfies the following conditions (F1) and (F2). 
\begin{itemize}
\item[(F1)] If $s=0$, % that is $k\equiv 2 \;({\rm mod}\; 4)$,  
\[f(x) = \left\{
\begin{array}{rl}
-1&\mbox{ if } x \in R_i \mbox{ for some } 1\le i\le m+1, \mbox{ and}\\
1&\mbox{ if } x \in T_i \mbox{ for some } 1\le i\le m.
\end{array}
\right.
\]
\item[(F2)] If $s=1$, let 
\begin{enumerate}
\item[(i)] $f(T_i) = |T_i| = b-r=\frac{k}{2}d$ for all $1 \le i \le m$;
\item[(ii)] $f(R_i) = d - \frac{k}{2} + 1$ for all $1 \le i \le m+1$.
\end{enumerate}
Moreover, let $R_i(1) = f^{-1}(1) \cap R_i$, for $1 \le i \le m+1$, and let $\mathcal{T}_i \subseteq R_{i} \cup T_i \cup R_{i+1}$ be the maximal block such that $f(\mathcal{T}_i)= |\mathcal{T}_i|$. Then let $f$ satisfy, additionally,
\begin{enumerate}
\item[(iii)] if $R_i(1)$ is not a block, then $|\mathcal{T}_i \setminus T_i| \ge d$, for $1 \le i \le m$;
\item[(iv)] if $R_i(1)$ is a block and $R_i \cap \mathcal{T}_i = \emptyset$, then $R_{i+1}(1)$ is a block and there are at most $\frac{k}{2}-1$ $(-1)$'s between $R_{i}(1)$ and $R_{i+1}(1)$.
\end{enumerate}
\end{itemize}
Finally, we set $\mathcal{F}^-_{d,k} = \{-f \;|\; f \in \mathcal{F}^+_{d,k}\}$ and $\mathcal{F}_{d,k} = \mathcal{F}^+_{d,k} \cup \mathcal{F}^-_{d,k}$. \\

For the sake of comprehension we shall see some particular examples. For  $d=2$ we analyze two cases corresponding to $s=0$ and $s=1$, namely $k=6$ and $k=8$.

\begin{example}
Let $d=2$ and $k=6$, thus $s=0$ and $n=12$. Then, the unique function $f \in \mathcal{F}^+_{2,6}$ is represented by
\[-\,-\,+\,+\,+\,+\,+\,+\,+\,+\,-\,-\]
where $R_1=\{1,2\}$, $T_1=\{3,\dots,10\}$ and $R_2=\{11,12\}$.
Note that, this function  satisfies $|f([n])| = 4=\frac{d-1}{d+1}n$ and has no ZS $(2,6)$-block.
\end{example}

\begin{example}
Let $d=2$ and $k=8$, thus $s=1$ and $n=18$. The functions from family $\mathcal{F}^+_{2,8}$ are described below. Here, $m=\frac{k-2s-2}{4}=1$, $b=(d+1)\frac{k}{2}+d-1=13$, $r=\frac{k}{2}-1+ds=5$, and so $R_1=\{1,\dots,5\}$, $T_1=\{6,\dots,13\}$ and $R_2=\{14,\dots,18\}$. We use capital letters to denote the different sequences of length $r=5$ and weight $d - \frac{k}{2} + 1=-1$:

\vspace{.4cm}
\hspace{1.3cm} A: $\,+\,+\,-\,-\,-$ \hspace{1.3cm} E: $\,+\,-\,+\,-\,-$  \hspace{1.3cm}  I: $\,-\,+\,-\,-\,+$

\hspace{1.3cm} B: $\,-\,+\,+\,-\,-$  \hspace{1.3cm} F: $\,+\,-\,-\,+\,-$  \hspace{1.3cm} J: $\,-\,-\,+\,-\,+$

\hspace{1.3cm} C: $\,-\,-\,+\,+\,-$  \hspace{1.3cm} G: $\,+\,-\,-\,-\,+$

\hspace{1.3cm} D: $\,-\,-\,-\,+\,+$  \hspace{1.3cm} H: $\,-\,+\,-\,+\,-$
\vspace{.3cm}

Accordingly to rules (F2) $(i)$ to $(iv)$, all the functions $f \in \mathcal{F}^+_{2,8}$ are represented by the following sequences:

\[X\,+\,+\,+\,+\,+\,+\,+\,+\,Y,\]
where
\[
Y \in \left\{\begin{array}{ll}
\{A\}, & \mbox{if } X \in \{A,E,F,H\};\\
\{A,B\}, & \mbox{if } X = B;\\
\{A,B,C\}, & \mbox{if } X = C;\\
\{A,B,C,D,E,F,G,H,I,J\}, & \mbox{if } X = D;\\
\{A,E,F,G\}, & \mbox{if } X \in \{G,I,J\}.\\
\end{array}
\right.
\]

Observe that all functions $f \in \mathcal{F}^+_{2,8}$ satisfy $|f([n])|= 6=\frac{d-1}{d+1}n$ and have no ZS $(2,8)$-blocks. 
\end{example}

Before stating the theorem showing sharpness for Theorem \ref{thm:(d,k)blocks_k>4}, we need to prove the following lemmas.
\begin{lemma}\label{la:connectedG}
Let $G$ be a connected graph with vertex set $V$, and, for integers $d \ge 1$ and $q$, let $g:V \longrightarrow d \mathbb{Z} + q$ be a function such that $|g(x) - g(y)| \in \{0, d\}$ for any pair of adjacent vertices $x, y \in V$. Then, for vertices $u, v \in V$ such that $g(u)  <  g(v)$ and for any $p \in d \mathbb{Z} + q$ with $g(u) \le p \le g(v)$, there exists a vertex $w \in V$ with $g(w) = p$.
\end{lemma}

\begin{proof}
Since $G$ is connected, there is a path $P = u_1 \ldots u_m$ with $u_1 = u$ and $u_m = v$. Let $p \in d \mathbb{Z} + q$ with $g(u) \le p \le g(v)$. If $p = g(u)$ or $p = g(v)$, we are done. So assume that $p$ is such that $g(u) < p < g(v)$. Let $i$ be the largest index such that $g(u_i) < p$. Then $p \le g(u_{i+1})$. On the other side, since $g(u_i) < p\le g(u_{i+1})$ and $g(u_i), p \in d \mathbb{Z} + q$, it follows that $d + g(u_i) \le p$. Hence, 
\[
g(u_{i+1}) = (g(u_{i+1}) - g(u_i)) + g(u_i) =|g(u_{i+1}) - g(u_i)| + g(u_i) \le d + g(u_i) \le p,
\]
and, thus, $g(u_{i+1}) = p$.
\end{proof}

\begin{lemma}\label{la:ZS-(k,d)-block}
Let $k$, $d$ and $n$ be positive integers such that $d \ge 2$ and $f:[n] \rightarrow \{-1,1\}$ a function on $[n]$. If there are $(k,d)$-blocks $S, T \subseteq [n]$ such that $f(S) < 0$ and $f(T) > 0$, then there is a $(k,d)$-block $B$ with $f(B) = 0$.
\end{lemma}

\begin{proof}
We define the graph $G=(V,E)$ as follows. Let $V$ be the set of all $(k,d)$-blocks of $[n]$. For, $B, B' \in V$, let $BB' \in E$ if and only if $|B \cap B'| = k-1$. Define the function $g:V \rightarrow 2\mathbb{Z}+q$ by $g(B) = f(B)$, where $q = 0$ if $k$ is even, and $q = 1$ if $k$ is odd. Then $|g(B) - g(B')| \in \{0,2\}$ for any pair of adjacent vertices $B, B' \in V$. It is not difficult to see that there is a way to transform, in a finite number of steps $\sigma$, one $(k,d)$-block $B$ into another $B'$ going through a sequence $B_0, B_1, \ldots, B_{\sigma}$ of different $(k,d)$-blocks, with $B_0 = B$ and $B_{\sigma} = B'$, such that consecutive members of the sequence differ by exactly one element, i.e. $|f(B_i)-f(B_{i+1})| \in \{0,2\}$ for $0 \le i \le \sigma-1$. This implies that the graph $G$ is connected. By Lemma \ref{la:connectedG}, if there are vertices $S, T \in V$ such that $g(S) < 0$ and $g(T) > 0$, then there is a vertex $B \in V$ with $g(B) = 0$. Hence, provided that $f(S) <0$ and $f(T) >0$, we have proven the existence of a $(k,d)$-block $B$ in $[n]$ with $f(B)=0$.
\end{proof}

Now we can prove the sharpness theorem.

\begin{theorem}\label{thm:sharpness_(d,k)blocks}
Let $k$, $d$ and $n$ be positive integers such that $k \ge 6$ is even, $d \ge 2$, and $n = \frac{d+1}{8}(k^2-2sk+4s-4)$, where $s \in \{0,1\}$ with $s \equiv \frac{k-2}{2} \;({\rm mod}\; 2)$. Then a function $f:[n] \rightarrow \{-1,1\}$ satisfies $|f([n])|  = \frac{d-1}{d+1}n$ and has no ZS $(d,k)$-blocks if and only if $f \in \mathcal{F}_{d,k}$.
\end{theorem}

\begin{proof}
We first prove that all functions $f \in \mathcal{F}^+_{d,k}$ satisfy $f([n])  = \frac{d-1}{d+1}n$. Recall that  $|T_i|=b-r$ for all $1\leq i\leq m$, and $f(T_i)=|T_i|$ in all cases. Therefore

\begin{equation}\label{eq:sum}
f([n]) =  \sum_{i=1}^m f(T_i)+\sum_{i=1}^{m+1}f(R_i) 
=m(b-r)+\sum_{i=1}^{m+1} f(R_i).
\end{equation}

If $s=0$ then $f(R_i)=-|R_i|=-r$, and so (\ref{eq:sum}) becomes

\begin{equation}\label{eq:s0}
f([n]) = m(b-r)-(m+1)r.
\end{equation}

Replacing in (\ref{eq:s0}) the values of $m$, $b$ and $r$ from (\ref{eq:mbr})  corresponding to $s=0$,  we obtain:

\begin{align*}
f([n]) &=\left(\frac{k-2}{4}\right)\left(\left(d+1\right)\frac{k}{2}+d-1-\left(\frac{k}{2}-1\right)\right)-\left(\frac{k-2}{4}+1\right)\left(\frac{k}{2}-1\right)\\
&=\left(\frac{k-2}{4}\right)\left(\frac{k+2}{2}\right)d-\left(\frac{k+2}{4}\right)\left(\frac{k-2}{2}\right)\\
&=\frac{d-1}{8}\left(k^2-4\right)\\
&=\frac{d-1}{d+1}n
\end{align*}

as claimed.\\

If $s=1$ then  $f(R_i) =\left(d - \frac{k}{2} + 1\right)$, and so (\ref{eq:sum}) becomes

\begin{equation}\label{eq:s1}
f([n]) = m(b-r)+(m+1)\left(d - \frac{k}{2} + 1\right).
\end{equation}

Replacing in (\ref{eq:s1}) the values of $m$, $b$ and $r$ from (\ref{eq:mbr})  corresponding to $s=1$,  we obtain:

\begin{align*}
f([n]) &=\left(\frac{k-4}{4}\right)\left(\left(d+1\right)\frac{k}{2}+d-1-\left(\frac{k}{2}+d-1\right)\right)+\left(\frac{k-4}{4}+1\right)\left(d-\frac{k}{2}+1\right)\\
&=\left(\frac{k-4}{4}\right)\left(\frac{k}{2}\right)d+\left(\frac{k}{4}\right)\left(d-\frac{k}{2}+1\right)\\
&=\frac{d}{8}\left(k^2-4k\right)+\frac{kd}{4}-\frac{k^2}{8}+\frac{k}{4}\\
&=\frac{d}{8}\left(k^2-2k\right)-\frac{2kd}{8}+\frac{kd}{4}-\frac{k^2}{8}+\frac{k}{4}\\
&=\frac{d}{8}\left(k^2-2k\right)-\frac{k^2-2k}{8}\\
&=\frac{d-1}{8}\left(k^2-2k\right)\\
&=\frac{d-1}{d+1}n
\end{align*}

as claimed.\\

Now we will prove that none of the functions $f \in \mathcal{F}^+_{d,k}$ contains a ZS $(d,k)$-block. As we previously used $R_i(1) = f^{-1}(1) \cap R_i$, let  $R_i(-1) = f^{-1}(-1) \cap R_i$ for every $1 \le i \le m+1$. We first prove that, in both cases $s=0$ and $s=1$, we have
\begin{equation}\label{eq:R-1}
R_i(-1)=\frac{k}{2}-1.
\end{equation}
To see this, recall that by definition, if $s=0$ then $R_i(-1)=|R_i|=r=\frac{k}{2}-1$ as claimed. For the case $s=1$, recall that $f(R_i)=d-\frac{k}{2}+1$ and, since $f(R_i)=R_i(1)-R_i(-1)$, the equality in (\ref{eq:R-1}) follows using the fact that $R_i(1)+R_i(-1)=|R_i|=r=\frac{k}{2}-1+d$.

Observe now that, in all sequences given by a function $f \in \mathcal{F}_{d,k}$, the $(-1)$'s are contained in the $R_i$ sets. In other words, $f^{-1}(-1)\cap T_i=\emptyset$ for every $1\leq i\leq m$. As a ZS $(d,k)$-block must contain the same number of $(-1)$'s and $(+1)$'s, namely $\frac{k}{2}$, then, by (\ref{eq:R-1}), a ZS $(d,k)$-block necessarily intersects two $R_i$ sets, which is impossible since, by construction, between the last $(-1)$ from $R_i$ and the first $(-1)$ from $R_{i+1}$ there are $d\left(\frac{k}{2}+1\right)$ $(+1)$'s.

Hence, we have proved that the functions $f \in \mathcal{F}^+_{d,k}$ satisfy $f([n]) = \frac{d-1}{d+1}n$ and have no ZS $(d,k)$-blocks and, thus, it follows that the functions $f \in \mathcal{F}_{d,k}$ fulfill $|f([n])| = \frac{d-1}{d+1}n$ and do not contain ZS $(d,k)$-blocks. Now we show the necessity of the statement.

%Let $k$, $d$ and $s$ be as in the statement of the theorem and let $n = \frac{d+1}{8}(k^2-2sk+4s-4)$. 

To show that the family $\mathcal{F}_{d,k}$ contains all functions $f:[n] \rightarrow \{-1,1\}$ with $|f([n])|  = \frac{d-1}{d+1}n$ and without ZS $(d,k)$-blocks, assume we have a function ${f}:[n] \longrightarrow \{-1,1\}$ such that ${f}([n])  = \frac{d-1}{d+1}n$ (the case ${f}([n])  = - \frac{d-1}{d+1}n$ is analogous) and such that it has no ZS $(k,d)$-blocks. By Lemma \ref{la:ZS-(k,d)-block}, all $(k,d)$-blocks have positive sum. Actually, because of $k$ being even, we may assume that all $(k,d)$-blocks have sum at least $2$.

Observe that
\[n = \frac{d+1}{8} (k^2-2sk+4s-4) = \frac{k-2s - 2}{4} \left((d+1) \frac{k}{2}+d-1 \right) + \frac{k}{2} +ds - 1.\]
Since $0 \le \frac{k}{2} +ds - 1 < (d+1) \frac{k}{2}+d-1$, the expression above represents the division of $n$ by $b = (d+1) \frac{k}{2}+d-1$ with rest $r = \frac{k}{2} +ds - 1$. Hence, we may divide the interval $[n]$ into $m = \frac{k-2s-2}{4}$ blocks $B_i$ of length $b$, with $1 \le i \le m$, and a rest block of length $r$. Moreover, each of the $b$-blocks $B_i$ will be tared into two parts: $R_i$ and $T_i$, where $|R_i| = \frac{k}{2}-1 = r$ and $|T_i| = d \left( \frac{k}{2}+1\right) = b-r$, where $1 \le i \le m$. The rest block of length $r$ will be called $R_{m+1}$. Formally, we have
\[R_i = \{(i-1)b + 1, \ldots, (i-1)b+r\}, \mbox{ and } T_{i'} = \{(i'-1)b+r+1, \ldots, i'b\},\]
where $1 \le i \le m+1$ and $1 \le i' \le m$.
Moreover, we define also the $b$-blocks $C_i = T_i \cup R_{i+1}$, for $1 \le i \le m$. Observe that 
\[[n] = \left(\bigcup_{i = 1}^{j-1} B_i\right) \cup R_j \cup \left(\bigcup_{i=j}^m C_i\right),\] for any $1 \le j \le m$.
\\

We will show first that any $b$-block has at most  $\frac{k}{2}-1$ $(-1)$'s. Suppose we have a $b$-block $B \subseteq [n]$ with a subset $S \subseteq B$ with $|S|=\frac{k}{2}$ and such that ${f}(s) = -1$ for all $s \in S$. Consider the set $B' = B \setminus S = \{b_1, b_2, \ldots, b_{|B'|}\}$, where $b_1 < b_2 < \ldots < b_{|B'|}$ and $|B'| = b - \frac{k}{2} = d \left(\frac{k}{2} +1\right) -1$. Then $S \cup \{b_d, b_{2d}, \ldots, b_{\frac{k}{2}d}\}$ is a $(k,d)$-block with sum at most $0$, which is not possible.

Having this, we can say now that ${f}(B_i), {f}(C_i) \ge d \left(\frac{k}{2}+1\right) - \frac{k}{2} + 1$ for all $1 \le i \le m$. Hence, we have, for any $1 \le j \le m+1$,

\begin{align*}
\frac{d-1}{8}(k^2-2sk+4s-4) = {f}([n]) &= \left(\sum_{i = 1}^{j-1} {f}(B_i)\right) + {f}(R_j) + \left(\sum_{i=j}^m {f}(C_i)\right)\\
& \ge \left( d \left(\frac{k}{2}+1\right) - \frac{k}{2} + 1 \right) m + {f}(R_j),
\end{align*}
which implies
\begin{align*}
{f}(R_j) &\le \frac{d-1}{8}(k^2-2sk+4s-4) - \left( d \left(\frac{k}{2}+1\right) - \frac{k}{2} + 1 \right) m\\
& = \frac{d-1}{8}(k^2-2sk+4s-4) - \left( d \left(\frac{k}{2}+1\right) - \frac{k}{2} + 1 \right) \frac{k-2s-2}{4}\\
& = \frac{d}{8} \left(k^2-2sk+4s-4 - (k+2)(k-2s-2)\right) \\
& \hspace{4ex}- \frac{1}{8} (k^2-2sk+4s-4 + (-k+2)(k-2s-2))\\
& = ds - \frac{k}{2} + 1.
\end{align*}
Now we have
\begin{align*}
d \left(\frac{k}{2}+1\right) - \frac{k}{2} + 1 &\le {f}(B_i) = {f}(R_i) + {f}(T_i) \\
&\le |T_i| + ds - \frac{k}{2} + 1\\
&= d \left(\frac{k}{2}+1-s \right) + ds - \frac{k}{2} + 1\\
&= d \left(\frac{k}{2}+1\right) - \frac{k}{2} + 1,
\end{align*}
which yields ${f}(T_i) = |T_i| = d \left(\frac{k}{2}+1-s \right)$ and ${f}(R_i) = ds - \frac{k}{2} + 1$. This fits to rule (F1) when $s=0$ and to rule (F2)(i) and (ii) when $s=1$. Hence, ${f}(t) = 1$ for all $t \in T_i$ and all $1 \le i \le m$, and $|R_i \cap {f}^{-1}(-1)| = \frac{k}{2}-1$ and $|R_i \cap {f}^{-1}(1)| = ds$ for all $1 \le i \le m+1$.
\\
If $s=0$, we have finished, as the only possible function ${f}$ is the following:

\[
{f}(x) = \left\{ \begin{array}{rl}
1, & \mbox{if } x \in T_i, 1 \le i \le m\\
-1, & \mbox{if } x \in R_i, 1 \le i \le m+1.
\end{array}
\right.
\]

So we assume from now on that $s=1$. We still have to show that $f$ satisfies rules (F2) (iii) and (iv). For simplicity, we define $R_i(1) = {f}^{-1}(1) \cup R_i$ and $R_i(-1) = {f}^{-1}(-1) \cup R_i$, for $1 \le i \le m+1$.\\

\noindent
{\it Claim 1: $f$ satisfies rule (F2)(iii).\\
Suppose that $R_i(1)$ is not a block.} Let $\mathcal{T}_i \subseteq R_i \cup T_i \cup R_{i+1}$ be a maximal block such that $T_i \subseteq \mathcal{T}_i$ and $f(\mathcal{T}_i) = |\mathcal{T}_i|$. Suppose that $|\mathcal{T}_i \setminus T_i| \le d-1$. Let $t$ be the first element of $R_{i+1}$ with ${f}(t) = -1$. Then $t = ib + \alpha$, for an integer $0 \le \alpha \le d-1$. Now consider the set $B = R_i(-1) \cup \{t-\frac{k}{2}d, \ldots, t-2d,t-d, t\}$. Let $t'$ be the last element from $R_i(-1)$. Since $T_i \subseteq \mathcal{T}_i$, $t' \ge t - |T_i| - d = t - \frac{k}{2}d-d$. Hence, $(t-\frac{k}{2}d) - t' \le d$ and, thus, $B$ is a $(k,d)$-block with $\frac{k}{2}$ $(-1)$'s and $\frac{k}{2}$ $1$'s, that is, a ZS $(k,d)$-block, which is a contradiction. Hence, $|\mathcal{T}_i \setminus T_i| \ge d$ and Claim 1 is satisfied.\\

\noindent
{\it Claim 2: $f$ satisfies rule (F2)(iv).\\
Suppose that $R_i(1)$ is a block and $R_i \cap \mathcal{T}_i = \emptyset$.} With this assumptions, $t = (i-1)b+\frac{k}{2}+d-1$, the last element from $R_i$, has $f(t) = -1$. Suppose now that $R_{i+1}(1)$ is not a block. Then, similarly as the case above, it is easy to see that $\{t, t+d, t+2d, \ldots, t+\frac{k}{2}d\} \cup R_{i+1}(-1)$ is a ZS $(k,d)$-block, a contradiction. Thus, $f$ satisfies rule (F2)(iv) and the claim is proved.\\

Hence, we have proved that ${f} \in \mathcal{F}_{d,k}$. Altogether, we have shown that a function $f:[n] \rightarrow \{-1,1\}$ satisfies $|f([n])|  = \frac{d-1}{d+1}n$ and has no ZS $(d,k)$-blocks if and only if $f \in \mathcal{F}_{d,k}$.
\end{proof}

\section{Decomposition into bounded gap sequences with bounded sum}

In this section, we will work in a more general framework dealing with $\{-r, s\}$-sequences, where $r$ and $s$ are positive integers. Our aim is to present a decomposition theorem of such sequences into bounded gap sequences of bounded sum. We will do this by means of a graph-theoretical approach. To this purpose, we shall give the following notation. \\

For positive integers $m$ and $n$, let $H_{n,m}$ be the directed graph consisting of the disjoint union of $m$ vertex sets $V_1, V_2, \ldots, V_m$ with $|V_i| =n$ for all $1 \le i \le m$, and arc set $E = \bigcup_{i=1}^{m-1} (V_i, V_{i+1})$, where $(V_i,V_{i+1}) = \{(v,w) \;|\; v \in V_i, w \in V_{i+1}\}$. Moreover, given positive integers $r, s$ and $m$, we define 
\[
L(r,s,m) := \{-rx+sy \;:\; x,y \in \mathbb{Z}, x,y \ge 0, x+y =m\},
\]
and, for any real number $q \in [-rm,sm]$,
\begin{align*}
\lambda(q,r,s,m) &:= \max\{p \in L(r,s,m) \;:\; p \le q\}, \mbox{ and}\\
\Lambda(q,r,s,m) &:= \min\{p \in L(r,s,m) \;:\; p \ge q\}.
\end{align*}

\begin{remark}\label{rem:lambda}
For positive integers $r,s,m$ and $q \in \mathbb{R}$, the following assertions hold.
\begin{enumerate}
\item[(i)] $\lambda(q,r,s,m) \le q \le \Lambda(q,r,s,m)$.
\item[(ii)] $\lambda(q,r,s,m) = \Lambda(q,r,s,m) = q$ if and only if $q \in L(r,s,m)$.
\item[(iii)] $\Lambda(q,r,s,m) - \lambda(q,r,s,m) \in \{0, r+s\}$.
\item[(iv)] For any positive integer $n$, 
\[
n\lambda(q,r,s,m) \le \lambda(qn,r,s,mn) \;\mbox{ and } \; n\Lambda(q,r,s,m) \ge \Lambda(qn,r,s,mn).
\]
\item[(v)] $L(r,s,m) = \{(r+s)y - rm \;:\; x \in [0,m]\cap \mathbb{Z}\} = \{-(r+s)x + sm \;:\; y \in [0,m]\cap \mathbb{Z}\}$.
\item[(vi)] $0 \in L(r,s,m)$ if and only if $\frac{r+s}{{\rm gcd}(r,s)}$ divides $m$.
\end{enumerate}
\end{remark}

\begin{proof}
Items (i), (ii) and (iii) follow straightforward from the definitions.\\
For (iv), let $p^* \in \{p \in L(r,s,m) \;:\; p \le q\}$. Then there are $x^*, y^* \ge 0$ with $x^*+y^* = m$ such that $p^* = -rx^*+sy^*$. Hence, $p^*n = -r(x^*n)+s(y^*n) \le qn$ with $x^*n+y^*n = mn$, and thus $p^*n \in \{p \in L(r,s,mn) \;:\; p \le qn\}$. This implies that
\[
n \cdot \{p \in L(r,s,m) \;:\; p \le q\} \subseteq \{p \in L(r,s,mn) \;:\; p \le qn\},
\]
which yields $n\lambda(q,r,s,m) \le \lambda(qn,r,s,mn)$. Analogously, we can show that 
\[
n \cdot \{p \in L(r,s,m) \;:\; p \ge q\} \subseteq \{p \in L(r,s,mn) \;:\; p \ge qn\},
\]
which gives $n\Lambda(q,r,s,m) \ge \Lambda(qn,r,s,mn)$.\\
Item (v) follows by replacing $x = m-y$ and, accordingly, $y = m-x$ for any element $p = -rx+sy \in L(r,s,m)$.\\
Lastly, for (vi), suppose first that $0 \in L(r,s,m)$. Then, with (v), we know that there is an $x \in \mathbb{Z} \cap [0,m]$ such that $0 = -(r+s) x + sm$, which is equivalent to 
\[
m = \frac{(r+s)x}{s} = \frac{r+s}{{\rm gcd}(r,s)} \cdot \frac{{\rm gcd}(r,s)x}{s}.
\]
Clearly, $\frac{r+s}{\rm gcd(r,s)} \in \mathbb{Z}$. To show that $\frac{r+s}{{\rm gcd}(r,s)}$ divides $m$, it remains to show that $\frac{{\rm gcd}(r,s)x}{s} \in \mathbb{Z}$, too. So let $a \in \mathbb{Z}$ be such that $r = a \cdot {\rm gcd}(r,s)$. Then, with $0 = -rx+sy$, we obtain
\[
sy = rx = a\cdot{\rm gcd}(r,s) \cdot x,
\]
and, since $1 = {\rm gcd}(a,s)$, it follows that $s$ divides ${\rm gcd}(r,s) x$. Hence, $\frac{{\rm gcd}(r,s)x}{s} \in \mathbb{Z}$. On the other side, if $\frac{r+s}{{\rm gcd}(r,s)}$ divides $m$, there is a $b \in \mathbb{Z}$ such $m = b \cdot \frac{r+s}{{\rm gcd}(r,s)}$. Then both, $x := \frac{sb}{{\rm gcd}(r,s)}$ and $y := \frac{rb}{{\rm gcd}(r,s)}$, are in $\mathbb{Z} \cap [0,m]$ and fulfill $x + y = m$. Thus,
\[
-r x + s y = -r \frac{sb}{{\rm gcd}(r,s)} + s \frac{rb}{{\rm gcd}(r,s)} = 0 \in L(r,s,m).
\]
\end{proof}

\begin{lemma}\label{la:H_nm}
Let $r, s, m , n$ be positive integers, and let $f: V(H_{n,m}) \rightarrow \{-r,s\}$ be a function on $V(H_{n,m})$. If there are directed $m$-paths $P$ and $Q$ in $H_{n,m}$ such that $f(V(P)) < f(V(Q))$, then, for any $p \in L(r,s,m)$ with $f(V(P)) \le p \le f(V(Q))$ there is a directed $m$-path $P^*$ such that $f(V(P^*)) = p$.
\end{lemma}

\begin{proof}

Let $G_{n,m}$ be the graph with vertex set
\[V(G_{n,m}) = \{P \;:\; P \mbox{ is a directed } m \mbox{-path in } H_{n,m}\}\]
and edge set
\[E(G_{n,m}) = \{PP' \;:\; |V(P) \cap V(P')| \ge m-1\}.\]
Observe first that $G_{n,m}$ is connected. Here for, we will show that, for any two vertices $P, P' \in V(G_{n,m})$, there is a path in $V(G_{n,m})$ joining them. Let $P = u_1u_2 \ldots u_m$ and $P' = v_1v_2 \ldots v_m$ in $H_{n,m}$. Let $R_0 = P$ and, for $1 \le i \le m$, let $R_i$ be the directed $m$-path consisting of the vertex set $(V(R_{i-1}) \cup\{v_i\}) \setminus \{u_i\}$. Then, clearly, $R_m = P'$ and, for $1 \le i \le m$, $|V(R_{i-1}) \cap V(R_i)| \ge m-1$. This implies that the set $\{R_0, R_1, \ldots, R_m\}$ induces a path in $G_{n,m}$ and, thus, we can conclude that $G_{n,m}$ is connected.

Let now 
\[g: V(G_{n,m}) \rightarrow (r+s)\mathbb{Z}-rm\] 
be a function defined by $g(P) = f(V(P))$. Observe that, for $P \in V(G_{n,m})$ and integers $x := |f^{-1}(-r) \cap V(P)|$ and $y := |f^{-1}(s) \cap V(P)|$, we have that $x + y = m$ and, together with Remark Remark \ref{rem:lambda} (v),
\[
g(P) = f(V(P)) = -rx + sy \in L(r,s,m) \subseteq (r+s)\mathbb{Z}-rm.
\]
In particular, $g$ is well defined.

Now let $P$ and $Q$ be directed paths in $H_{n,m}$ such that $g(P) < g(Q)$ and take $p \in L(r,s,m)$ with $g(P) \le p \le g(Q)$.  Then Lemma \ref{la:connectedG} yields the existence of a directed $m$-path $P^*$ with $ f(V(P^*)) = g(P^*) = p$.
\end{proof}

\begin{theorem}\label{thm:decomposition}
Let $r, s, m , n$ be positive integers, let $f: V(H_{n,m}) \rightarrow \{-r,s\}$ be a function on $V(H_{n,m})$, and let $q = \frac{1}{n} f(V(H_{n,m}))$. Then $H_{n,m}$ can be decomposed into $n$ vertex-disjoint directed $m$-paths $P_1, P_2, \ldots, P_n$ such that 
\[\lambda(q,r,s,m) \le f(V(P_i)) \le \Lambda(q,r,s,m),\]
for all $1 \le i \le n$.
\end{theorem}

\begin{proof}
We will prove the statement by induction on $n$. If $n = 1$, then $H_{n,m} = H_{1,m}$ is itself a directed path of length $m$ with
\[
\lambda(f(V(H_{1,m}),m) = f(V(H_{1,m})) = \Lambda(f(V(H_{1,m}),m),
\]
and we are done. 

Assume now that the theorem is true for $n$. We will prove the statement  for $n+1$. Let $f: V \rightarrow \{-r,s\}$ be a function on $V = V(H_{n+1,m})$, and let $q = \frac{1}{n+1} f(V)$. We will write, for short, $\lambda(q, m) = \lambda(q,r,s,m)$ and $\Lambda(q, m) = \Lambda(q,r,s,m)$. We will show in the following that there is always a directed $m$-path $P$ such that $\lambda(q,m) \le f(V(P)) \le \Lambda(q,m)$ and such that $\lambda(q,m) = \lambda(q',m)$ and $\Lambda(q,m) = \Lambda(q',m)$ for $q' = \frac{1}{n} (f(V) - f(P))$. This will allow us to use the induction hypothesis.  

Consider any decomposition of $V$ into $n+1$ vertex disjoint $m$-paths $P_1, P_2, \ldots, P_{n+1}$ and observe that $f(V) \in L(r,s,m(n+1))$. If $f(V(P_i)) \in \{ \lambda(q,m), \Lambda(p,q) \}$ for all $1 \le i \le n+1$, then we have already the desired decomposition. So suppose this is not the case. We will show that there are paths $P^*$ and $Q^*$ with $f(V(P^*)) = \lambda(q, m)$ and $f(V(Q^*)) = \Lambda(q, m)$. \\
If $f(V(P_i)) \le \lambda(q, m)$ for all $1 \le i \le n+1$, then, since not all $f(V(P_i))$ can be equal to $\lambda(q, m)$, we obtain, applying Remark \ref{rem:lambda} (ii) and (iv),
\[
f(V) = \sum_{i=1}^{n+1} f(V(P_i)) < (n+1) \lambda(q,m) \le \lambda(q(n+1),m(n+1)) = \lambda(f(V) , m (n+1)) = f(V),
\]
which is a contradiction. In the same way, if $f(V(P_i)) \ge \Lambda(q, m)$ for all $1 \le i \le n+1$, we obtain the contradiction
\[
f(V) = \sum_{i=1}^{n+1} f(V(P_i)) > (n+1) \Lambda(q,m) \ge \Lambda(q(n+1),m(n+1)) = \Lambda(f(V) , m (n+1)) = f(V).
\]
Hence, there have to exist paths $P$ and $Q$ such that $f(V(P)) > \lambda(q, m)$ and $f(V(Q)) < \Lambda(q, m)$. Since this means that $f(V(P)) \ge \Lambda(q, m)$ and $f(V(Q)) \le \lambda(q, m)$, we obtain with Lemma \ref{la:H_nm} that, for any $p \in L(r,s,m)$ with $\lambda(q,m) \le p \le \Lambda(q,m)$, there exists a directed $m$-path with $f$-value $p$. In particular, it follows that there are directed $m$-paths $P^*$ and $Q^*$ with $f(V(P^*)) =  \lambda(q, m)$ and $f(V(Q)) = \Lambda(q, m)$. Now we distinguish two cases.\\

\noindent
{\it Case 1. Suppose that $q - \lambda(q,m) \le \Lambda(q,m) - q$.}\\
We will show that, in this case, $P^*$ is the appropriate path to be removed in order to apply the induction hypothesis. So let $q' = \frac{1}{n} (f(V) - f(V(P^*))$. Then
\[
q' =  \frac{f(V) - f(V(P^*)}{n} = \frac{f(V) - \lambda(q, m)}{n} = \frac{q(n+1)}{n} - \frac{\lambda(q,m)}{n} = q + \frac{q-\lambda(q,m)}{n}.
\] 
Thus, using the assumption $q - \lambda(q,m) \le \Lambda(q,m) - q$ and Remark \ref{rem:lambda} (i),
\begin{align*}
q' & = q + \frac{q-\lambda(q,m)}{n} \\
&\le q + \frac{\Lambda(q,m)-q}{n}\\
& = \frac{q(n-1)+ \Lambda(q,m)}{n} \\
&\le \frac{\Lambda(q,m) (n-1) + \Lambda(q,m)}{n} \\
& = \Lambda(q,m).
\end{align*}
On the other side, we have $q' = q + \frac{q-\lambda(q,m)}{n} \ge q \ge \lambda(q,m)$. Hence, we have
\[
\lambda(q',m) = \lambda(q,m) \; \mbox{ and } \; \Lambda(q',m) = \Lambda(q,m).
\]
Moreover, since $H_{n,m} \cong H_{n+1,m} - V(P^*)$, we can apply here the induction hypothesis. Hence, as $q' = \frac{f(H_{n+1,m} - V(P^*))}{n} = \frac{f(V) - V(P^*)}{n}$, there is a decomposition of $H_{n+1,m} - V(P^*)$ into $n$ disjoint $m$-paths $P_1, P_2, \ldots, P_n$ such that 
\[
\lambda(q,m) = \lambda(q',m) \le f(V(P_i)) \le \Lambda(q',m) = \Lambda(q,m),
\]
for $1 \le i \le n$. Hence, together with the path $P^*$, we obtain the desired decomposition of $H_{n+1,m}$.\\
\noindent
{\it Case 2. Suppose that $q - \lambda(q,m) \ge \Lambda(q,m) - q$.}\\
In this case, it will follow that $Q^*$ is the appropriate path to be removed in order to apply the induction hypothesis. So let $q' = \frac{1}{n} (f(V) - f(V(Q^*))$. Then
\[
q' =  \frac{f(V) - f(V(Q^*))}{n} = \frac{f(V) - \Lambda(q, m)}{n} = \frac{q(n+1)}{n} - \frac{\Lambda(q,m)}{n} = q + \frac{q-\Lambda(q,m)}{n}.
\] 
Thus, using the assumption $q - \lambda(q,m) \ge \Lambda(q,m) - q$ and Remark \ref{rem:lambda} (i),
\begin{align*}
q' & = q + \frac{q-\Lambda(q,m)}{n} \\
&\ge q + \frac{\lambda(q,m)-q}{n}\\
& = \frac{q(n-1)+ \lambda(q,m)}{n} \\
&\ge \frac{\lambda(q,m) (n-1) + \lambda(q,m)}{n} \\
& = \lambda(q,m).
\end{align*}
On the other side, we have $q' = q + \frac{q-\Lambda(q,m)}{n} \le q \le \Lambda(q,m)$. Hence, we have again
\[
\lambda(q',m) = \lambda(q,m) \; \mbox{ and } \; \Lambda(q',m) = \Lambda(q,m).
\]
Moreover, since $H_{n,m} \cong H_{n+1,m} - V(Q^*)$, we can apply here the induction hypothesis on this graph. Hence, with $q' = \frac{f(H_{n+1,m} - V(Q^*))}{n} = \frac{f(V) - V(Q^*)}{n}$, there is a decomposition of $H_{n+1,m} - V(Q^*)$ into $n$ disjoint $m$-paths $Q_1, Q_2, \ldots, Q_n$ such that 
\[
\lambda(q,m) = \lambda(q',m) \le f(V(Q_i)) \le \Lambda(q',m) = \Lambda(q,m),
\]
for $1 \le i \le n$. Hence, together with the path $Q^*$, we obtain the desired decomposition of $H_{n+1,m}$.
\end{proof}

By Lemma \ref{rem:lambda} (vi), given positive integers $r, s, m$, and $n$, $0 \in L(r,s,m)$ if and only if $\frac{r+s}{\gcd(r,s)}$ divides $m$. In such a case, if we are provided with a function $f: V(H_{n,m}) \rightarrow \{-r,s\}$ such that $f(V(H_{n,m})) = 0$, it follows by previous theorem that the graph $H_{n,m}$ can be decomposed into $n$ $m$-paths of sum zero. Observe that such a ZS $m$-path has $\frac{qr}{\gcd(r,s)}$ vertices with $f$-value $s$ and $\frac{qs}{\gcd(r,s)}$ vertices with $f$-value $-r$, where $q = \frac{\gcd(r,s)}{r+s}m$. We have thus the following corollary.

\begin{corollary}\label{cor:decomp_zs}
Let $r, s, m , n$ be positive integers such that $\frac{r+s}{\gcd(r,s)}$ divides $m$. If $f: V(H_{n,m}) \rightarrow \{-r,s\}$ is a function such that $f(V(H_{n,m})) = 0$, then $H_{n,m}$ can be decomposed into $n$ vertex-disjoint directed ZS $m$-paths $P_1, P_2, \ldots, P_n$.
\end{corollary}

\mbox{}\\[-4ex]

The case where  $r=s=1$ is also of our interest.

\begin{corollary}\label{cor:deconp_1s}
Let $m , n$ be positive integers, and let $f: V(H_{n,m}) \rightarrow \{-1,1\}$ be a function on $V(H_{n,m})$ and let $q = \frac{f(V(H_{n,m}))}{n}$ and $k = \lceil q \rceil$. Then $H_{n,m}$ can be decomposed into $n$ vertex-disjoint directed $m$-paths $P_1, P_2, \ldots, P_n$ with  
\[
f(V(P_j)) \in \left\{
\begin{array}{ll}
\{k - 2, k\}, & \mbox{if } m \equiv k \;({\rm mod} \; 2)\\
\{k-1, k+1\} & \mbox{if } m \not\equiv k \;({\rm mod} \; 2),
\end{array}
\right.
\]
for all $1 \le j \le n$.
\end{corollary}

\begin{proof}
Let $\lambda(q,m) = \lambda(q,1,1,m)$ and $\Lambda(q,m) = \Lambda(q,1,1,m)$. By Theorem \ref{thm:decomposition}, there is a decomposition of $H_{n,m}$ into $n$ directed $m$-paths $P_1, P_2, \ldots, P_n$ such that $f(V(P_j)) \in \{\lambda(q,m), \Lambda(q,m)\}$. If $q = \lambda(q,m)$ or $q = \Lambda(q,m)$, then $q \in \mathbb{Z}$ and thus $k = q$. Since $\sum_{1 \le j \le n} f(V(P_j)) = f(V_{n,m}) = qn$, it follows that in either case $f(V(P_j)) = q = k$, for $1 \le j \le n$, and we are done. So we may suppose that $\lambda(q,m) < q < \Lambda(q,m)$. This implies that $k = \lceil q \rceil \in \{\Lambda(q,m) -1, \Lambda(q,m)\}$. Further, since, for each $1 \le j \le n$, there is an $x_j \in [0,m] \cap \mathbb{Z}$ such that $f(P_j) = m - 2x_j \in \{\lambda(q,m), \Lambda(q,m)\} = \{\Lambda(q,m)-2, \Lambda(q,m)\}$ (see Remark \ref{rem:lambda}(v)), we have that $m \equiv \Lambda(q,m)$. Hence, we have two cases: either $m \equiv k = \lceil q \rceil = \Lambda(q,m)$ or $m \not\equiv k = \lceil q \rceil = \Lambda(q,m) -1$. In the first case, we obtain $f(P_j) \in \{\Lambda(q,m)-2, \Lambda(q,m)\} = \{ k-2, k\}$ and in the second $f(P_j) \in \{\Lambda(q,m)-2, \Lambda(q,m)\} = \{ k-1, k+1\}$, as claimed.
\end{proof}

From Theorem \ref{thm:decomposition}, we can deduce the following result. Here we consider a partition $I_1\cup I_2\cup...\cup I_m$ of the interval  $[nm]$, where $I_i=\{(i-1)n+1, (i-1)n+2, \ldots, in\}$, for $1 \le i \le m$. 

\begin{theorem}\label{thm:decomposition_nm}
Let $r, s, m , n$ be positive integers, let $f: [nm] \rightarrow \{-r,s\}$ be a function on $[nm]$, and let $q = \frac{1}{n} f([nm])$. Then $[nm]$ can be decomposed into $n$ disjoint $(2n-1)$-bounded gap sequences $S_1, S_2, \ldots, S_n$ each of cardinality $m$, such that 
\[
\lambda(q,r,s,m) \le f(S_j) \le \Lambda(q,r,s,m),
\]
where $|S_j \cap I_i| =1$ for all $1 \le j \le n$ and all $1 \le i \le m$.
\end{theorem}

\begin{proof}
Consider $I_i$ as the vertex partition set $V_i$ of the graph $H_{n,m}$ and apply Theorem~\ref{thm:decomposition} in order to obtain the decomposition of $H_{n,m}$ into directed $m$-paths $P_1, P_2, \ldots, P_n$ such that $\lambda(q,r,s,m) \le f(V(P_j)) \le \Lambda(q,r,s,m)$, for all $1 \le j \le n$. Then, setting $S_i = V(P_i)$, we have $\lambda(q,r,s,m) \le f(S_j) \le \Lambda(q,r,s,m)$ and $|S_j \cap I_i| =1$ for all $1 \le j \le n$ and all $1 \le i \le m$ and so each $S_j$ is a $(2n-1)$-bounded gap sequence. Hence, we obtain the desired decomposition of $[nm]$.
\end{proof}

The following corollaries point out the particular cases of Theorem \ref{thm:decomposition_nm} where $\frac{r+s}{\gcd(r,s)}$ divides $m$ and where $r=s=1$.

\begin{corollary}
Let $r, s, m , n$ be positive integers such that $\frac{r+s}{\gcd(r,s)}$ divides $m$. If $f: [nm] \rightarrow \{-r,s\}$ is a function such that $f([nm]) =0$, then $[nm]$ can be decomposed into $n$ disjoint ZS $(2n-1)$-bounded gap sequences $S_1, S_2, \ldots, S_n$ each of cardinality $m$, such that $|S_j \cap I_i| =1$ for all $1 \le j \le n$ and all $1 \le i \le m$.
\end{corollary}
\mbox{}\\[-4ex]

\begin{corollary}
Let $m , n$ be positive integers, let $f: [nm] \rightarrow \{-1,1\}$ be a function on $[nm]$ and let $q = \frac{f([nm])}{n}$ and $k = \lceil q \rceil$. Then $[nm]$ can be decomposed into $n$ disjoint $(2n-1)$-bounded gap sequences $S_1, S_2, \ldots, S_n$ each of cardinality $m$,  such that 
\[
f(V(S_j)) \in \left\{
\begin{array}{ll}
\{k - 2, k\}, & \mbox{if } m \equiv k \;({\rm mod} \; 2)\\
\{k-1, k+1\} & \mbox{if } m \not\equiv k \;({\rm mod} \; 2),
\end{array}
\right.
\]
and $|S_j \cap I_i| =1$ for all $1 \le j \le n$ and all $1 \le i \le m$.
\end{corollary}
\mbox{}\\[-4ex]

\section{Simple applications and open problems}

In this concluding section, we shall give an example of two possible applications of our results to number-theoretical questions, and we will pose a few problems which we found particularly interesting.

\subsection{Applications}

Part of our motivation to study the problems mentioned in this paper is the obvious relation to well-known number-theoretical functions. We give below two such applications which are consequences of Theorem \ref{thm:o(n)}.\\

\noindent
\textbf{Application 1.} Our first application deals with the Liouville function in a direction inspired by the work of Hildebrand on sign patterns of this function in short intervals \cite{Hil}. Let $\lambda$ denote the Liouville function defined as $\lambda(n)=(-1)^{\Omega(n)}$ where $\Omega(n)$ denotes the number of prime factors of $n$ (counted with multiplicity). In \cite{Hil} Hildebrand shows that, for every pattern $(\epsilon_1, \epsilon_2, \epsilon_3)$, where $\epsilon_i\in\{-1,1\}$ for $i\in \{1,2,3\}$, there are infinitely many positive integers $n$ such that $\lambda(n+i)=\epsilon_i$. In other words, for any $\{-1,1\}$-pattern of length three there are infinitely many $3$-blocks in the  Liouville function sequence satisfying such pattern. Next we prove that, for any $k\geq 2$, there are infinitely many zero-sum $k$-blocks in the  Liouville function sequence.

It is well-known that the Prime Number Theorem is equivalent to the estimate $\sum_{i=1}^n\lambda(i)=o(n)\leq  \frac{n}{\omega (n)}$ for some function $\omega(n)$ such that $\omega(n)\to \infty$ with $n$ (see \cite{BoChCo}). Thus, the following corollary is a direct consequence of Theorem \ref{thm:o(n)}.

\begin{corollary}\label{col:liouville}
For every $k\geq 2$, $k\equiv 0$ {\rm (mod $2$)}, there are infinitely many positive integers $n$, such that $\sum_{i=0}^{k-1} \lambda(n+i)=0$.
\end{corollary}

Moreover, since $\lambda$ is a multiplicative function, i.e. $\lambda (mn)= \lambda (m)\lambda (n)$ for any positive integers $n, m$, Corollary~\ref{col:liouville} can be extended to arithmetic progressions. This follows from the observation that $\sum_{i \le n} \lambda(id) = \lambda(n) \sum_{i \le n} \lambda(i) = o(n) = \frac{n}{\omega(n)}$ for some function $\omega$ fulfilling $\omega(n) \rightarrow \infty$ with $n$.

\begin{corollary}\label{col:liouville2}
For every $k\geq 2$, $k\equiv 0$ {\rm (mod $2$)}, and, for every $d\geq 1$,  there are infinitely many arithmetic progressions $\{a_1, a_2,...,a_k\}$ of length $k$ and common difference $d$ such that $\sum_{i=0}^{k-1} \lambda(a_i)=0$.
\end{corollary}

\noindent
\textbf{Application 2.} Our second application deals with the Legendre function and zero-sum blocks of consecutive primes. Given a prime number $p\geq 3$, the Legendre symbol (of quadratic versus non-quadratic residue modulo $p$) is a multiplicative function defined as:

\[\left( \frac{a}{p}\right) = \left\{
\begin{array}{rl}
1, &  \mbox{if $a$ is a quadratic residue (mod $p$) and $a\not\equiv 0$ (mod $p$)},\\
-1, &  \mbox{if $a$ is a non-quadratic residue (mod $p$)},\\
0, & \mbox{if $a\equiv 0$ (mod $p$)}.
\end{array}
\right.
\]

Restricted to $\mathbb{P} = \{p_1, p_2, \ldots,p_n, \ldots\}$, the set of prime numbers, we obtain a $\{-1,1\}$-valued function $f$ via the rule $f(p_j)=\left( \frac{p_j}{p}\right)$.

By the prime number theorem  for arithmetic progressions \cite{Hux}, the primes are evenly distributed among the residue classes modulo $p$. To be precise, the number is crudely  $\pi(a,b,n) = |\{ ax +b \in \mathbb{P}\;|\;  1 \le x \le n \}| =  \frac{(1+o(1))n}{\Phi(a)\ln(n)}$, where $\Phi(a)$ is the Euler function and ${\rm gcd}(a,b) = 1$. In particular, up to $p_n$ there are  $\frac{n}{p-1} + o(\frac{n}{p-1})$ primes in any of the sequences  of the form $px +j$, where $j\in\{1,...,p-1\}$ (covering all the residue classes modulo $p$, except for $0$ (mod $p$)). 
Hence,  $F(n ,p):= \left|\sum_{j=1}^{n} f(p_ j) \right|=\left|\sum_{j=1}^{n} \left( \frac{p_j}{p}\right) \right|=o(n)\leq  \frac{n}{\omega (n)}$ for some function $\omega(n)$ such that $\omega(n)\to \infty$ with $n$. So, we can apply Theorem \ref{thm:o(n)} on blocks of consecutive primes to obtain infinitely many zero-sum $k$-blocks in the sequence of primes. More precisely: 

\begin{corollary}\label{col:legendre}
For any prime  $p\geq 3$  and even integer $k\geq 2$, there are  infinitely many positive integers $n$ such that $\sum_{i=0}^{k-1} f(p_{n+i})=0$.
\end{corollary}

\subsection{Open problems}

There are many possible natural generalizations of our main theorem about zero-sum $k$-blocks when the range  of $f:[n]\to\{-1 ,1\}$ is replaced by various choices  of other ranges, such as $[x,y]=\{x,x+1,...,y\}$ and alike.  However, in most cases the existence of precisely $k$-consecutive zero-sum terms is not guaranteed as there is no corresponding ``interpolation lemma" (Lemma \ref{lem:k-blocks-t-sum}). Yet, in the case where $f: [n]\to \{ -r ,s \}$ for arbitrary positive integers $r$ and $s$, similar results can be deduced with more efforts. In particular, it should  be possible to prove the following conjecture along the same lines.

\begin{conjecture} 
Let  $r$, $s$ and $k$ be positive integers such that $(r+s)$ divides $k$. Then, there exists $c(r,s)$ satisfying that, if $n\geq \frac{rs}{(r+s)^2}k^2+c(r,s)k$ then
any function $f:[n]\to\{-r,s\}$ with $f([n])=0$ contains a ZS $k$-block.
\end{conjecture}

\begin{problem} 
Determine the minimum value of $N(r,s,k)$ such that  for $n\geq N(r,s,k)$  the above statement holds. 
\end{problem}

Another natural candidate for generalization is to replace the structure of blocks or bounded gap subsequences with that of arithmetic progressions. This leads us to the next problem.

\begin{problem} 
Let  $r$, $s$ and $k$ be positive integers such that $(r+s)$ divides $k$. Determine the minimum value $M(r,s,k)$ such that, if  $n\geq  M(r,s,k)$, then, for every function $f:[n]\to\{-r,s\}$ with $f([n])=0$, there is a $k$-term  arithmetic progression $A \subseteq [n]$ with $f(A)= 0$.  
\end{problem}
Clearly, the exact bounds obtained here for the case of zero-sum $k$-blocks serve as an upper bound for $M(r,s,k)$. In particular, $M(r,s,k) \leq c(r,s)k^2$ for some function $c(r,s)$. However, already for the case $\{ -r ,s\}  = \{ -1 ,1\}$, there are values of $k$ where the constructions forbidding $k$-blocks does not forbid  $k$-term  arithmetic progressions. 

\begin{proposition}\label{prop:arithm-prog}
Let $k\equiv 2$ (mod $4$) and $n=\frac{1}{4}k^2-1$. If $\frac{k}{2}$ is not a prime, then any function $f:[n]\to \{-1,1\}$ with  $f([n])=0$ contains a ZS $k$-term  arithmetic progression.
\end{proposition}

\begin{proof}
Consider $f:[n]\to \{-1,1\}$ with  $f([n])=0$. By Theorem \ref{thm:sharpness_k-blocks-t-sum}, we know that $f$ either contains a ZS $k$-block (therefore, a ZS-$k$-term  arithmetic progression) or is exactly the function defined below. Let  $m$ and $r$ be integers such that $n = km+r$, $m \ge 0$ and $0 \le r \le k-1$, then

\[f(x) = \left\{
\begin{array}{rl}
-1, &  \mbox{if } x \in R,\\
1, & \mbox{otherwise},
\end{array}
\right.
\]
where $R =R_1\cup R_2 \cup .... \cup R_{m+1}$ and  $R_i= \{(i-1)k+1,(i-1)k+2, \ldots, (i-1)k+r\}$, for $i \in [m]$. 
Thus, $f$ corresponds to a sequence $R_1\, P_1\, R_2\, P_2\,.... \,R_m\, P_m\, \, R_{m+1}$, where each $R_i$ is a $(\frac{k}{2}-1)$-block of $-1$'s, and each $P_i$ is a $\left(\frac{k}{2}+1\right)$-block of $1$'s. For any divisor $d$ of $\frac{k}{2}$ such that $d<\frac{k}{2}$, consider the set \[A=\{x \in [n] \;|\;   x\equiv 1  \mbox{(mod $d$)}  \mbox{ and }  x< (k+1)d\}.\]
Note that $A$ is a $k$-term  arithmetic progression with common difference $d$, and  $f(A)=0$ since $A$ takes $\frac{k}{2d}$ terms of each $B_i$ and $\frac{k}{2d}$ terms of each $P_i$ that it intersects.

\end{proof}

Yet it seems the following holds: 

\begin{conjecture}\label{conj:arithm-prog}
There are positive constants $c_1 = c_1(r ,s)$, $c_2=c_2(r,s)$, such that  \[c_1k^2 \leq M(r,s,k) \leq c_2k^2.\]
\end{conjecture}

%--------------------------------------------
\section{Acknowledgements}
%--------------------------------------------

The second author was partially supported by PAPIIT IA103915 and CONACyT project 219775. The third author was partially supported by PAPIIT IN114016 and CONACyT project 219827. Finally, we would like to acknowledge the support from Center of Innovation in Mathematics, CINNMA A.C.

%%%%%%%%%%%%%%%%%%%%%%%%%%%%%%%%%%%%%%%%%%%%%%%%%%%%%%%

\end{document}